\documentclass[final,reqno]{siamltex}
\usepackage{latexsym,amsmath,amssymb,amsfonts,mathrsfs}
\usepackage{epsf,graphicx,epsfig,color,cite,cases}
\usepackage{subfigure,graphics,multirow,marginnote,enumerate,bm}
\newcommand{\Grad}{{\rm Grad\,}}

\newcommand{\Om}{\Omega}

\newcommand{\pa}{\partial}

\newcommand{\ov}{\overline}
\newcommand{\I}{{\rm Im}}
\newcommand{\Rt}{{\rm Re}}

\newcommand{\wid}{\widetilde}

\newcommand{\na}{\nabla}
\newcommand{\mat}{\mathbb}

\newcommand{\R}{{\mat R}}

\newcommand{\N}{{\mat N}}
\newcommand{\C}{{\mat C}}
\newcommand{\Sp}{{\mat S}}

\newcommand{\be}{\begin{eqnarray}}
\newcommand{\ben}{\begin{eqnarray*}}
\newcommand{\en}{\end{eqnarray}}
\newcommand{\enn}{\end{eqnarray*}}

\newtheorem{remark}[theorem]{Remark}

\begin{document}
\renewcommand{\theequation}{\arabic{section}.\arabic{equation}}

\title{\bf direct and inverse acoustic scattering by global rough surfaces}
\author{Chengyu Wu\thanks{School of Mathematics and Statistics, Xi'an Jiaotong University,
Xi'an 710049, Shaanxi, China ({\tt wucy99@stu.xjtu.edu.cn})}
\and
Jiaqing Yang\thanks{School of Mathematics and Statistics, Xi'an Jiaotong University,
Xi'an 710049, Shaanxi, China ({\tt jiaq.yang@xjtu.edu.cn})}
}
\date{}
\maketitle


\begin{abstract}
  In this paper, we investigate on the forward and inverse acoustic scattering problem by an infinite rough interface in a lossless medium for both the cases that the real transmission coefficient $\mu\neq1$ and $\mu=1$. We first prove the existence of the unique solution in the bounded and continuous function space utilizing the integral equation method through an elaborate analysis, which also leads to the $L^p$ well-posedness. Then we carefully consider the singularity of the solutions to the problem with incident point source or hypersingular point source, where a simple and novel perspective is given for the derivation of the singularity. Finally, a global uniqueness theorem for the inverse problem is proven on the unique determination of the unbounded rough surface, the transmission coefficient and the wave number in the lower half plane from the measurements of the near field only on a line segment above the interface at a fixed frequency. 
\end{abstract}



\pagestyle{myheadings}
\thispagestyle{plain}
\markboth{C. Wu and J. Yang}{Direct and inverse scattering by unbounded penetrable rough surfaces}

\section{Introduction}\label{sec1}
\setcounter{equation}{0}
The present paper concerns the direct and inverse problem of scattering of time harmonic waves by an infinite rough interface. A transmission condition is imposed on the rough surface, and the medium above and below the rough surface are assumed to be homogeneous and isotropic. By the phrase rough surface, we denote a surface which is a (usually nonlocal) perturbation of an infinite plane surface such that the whole surface lies within a finite distance of the original plane. Rough surface scattering problems have arisen more and more interests in recent years due to their large number of applications in various fields, e.g., in optics, acoustics, geophysical exploration and underwater detection, and have been studied extensively through numerical and analytic methods (see, e.g., \cite{GP13,GP14,SC05,JA02,JGR16,XBH18,JA91,AG94,KW01} and the references quoted there). 

Till now, despite the numerous literatures on numerical algorithms, much less rigorous mathmatical results concerning the rough surface have been carried out. The strict analysis on the direct rough surface scattering problems can date back to the end of the last centery. 
In 1996, Chandler-Wilde and Ross \cite{SC96} considered the Dirichlet boundary value problem for the Helmholtz equation in a nonlocally perturbed half-plane and proposed a new boundary integral equation formulation using the impedance Green's function for an half-plane. They proved that the integral equation is well-posed in the space of bounded and continuous functions with the boundary being both Lyapunov and a small perturbation of a flat boundary. Later in \cite{SB98}, Chandler-Wilde and Zhang proposed a radiation condition called the upward propagating radiation condition for the two-dimensional rough surface scattering problem, which generalizes both the classical Sommerfeld radiation condition and the Rayleigh expansion condition for diffraction gratings. They studied the aforementioned Dirichlet problem and proved the uniqueness of the solution when the boundary is piecewise Lyapunov. In addition, based on the previous results \cite{SC96}, they showed the existence of the solution satisfying the new radiation condition in the case that the whole surface is both Lyapunov and a small perturbation of a flat boundary. Chandler-Wilde, Ross and Zhang \cite{SCB99} then further proved the well-posedness of this Dirichlet problem with the upward propagating radiation condition for general Lyapunov rough surfaces and all wavenumbers empolying the results \cite{SBC00} on the solvability of integral equations on the real line. From then on, a rather complete theory about the integral equation method for two-dimensional rough surface scattering problem was established. Many other models concerning the rough surface and inhomogeneous layer were investigated (see, e.g.,  \cite{SNB98,BS98,SB99,DTS03,BS03}). 
These work showed the well-posedness of the corresponding problem through the integral equation method in a similar manner, which is followed by the present study, where the uniqueness was obtained by establishing the integral representation and some relating basic inequalities for the solutions and utlizing the technical lemma \cite[Lemma A]{SB98}, and the existence proof depended on the general results on the solvability of weakly singular second-kind integral equations on unbounded regions in \cite{SB97,SB02,SBC00}. It is noted that the same procedure can not deal with the three-dimensional rough surface scattering due to the slow decay at infinity of the fundamental solution of the three-dimensional Helmholtz equation, and we refer to \cite{SE06a,SE06b} for the integral equation method in the three-dimensional case. Moreover, the variational methods were developed in \cite{SJ10,SP05,SPM07,MT06,GXFB15,AS10,PJ12,GHPL19}, where the radiation condition angular spectrum representation is usually imposed, i.e., the solutions can be represented in integral form as a superposition of upward traveling and evanescent plane waves, which is known to be equivalent to the upward propagating radiation condition. We also refer to \cite{TT05} for some further discussion about these radiation conditions. 

It should be noticed that the majority of the preceding work are concerning the half-plane problems, while less attention has been paid to the whole space problems, which require a substantially more elaborate analysis due to the presence of the transmitted wave.  Particularly, \cite{SB99,AS10,MT06} were concerned with the scattering by inhomogeneous layer embedded in the whole space. In \cite{PJ12}, Li and Shen studied the scattering by a penetrable rough surface in a lossy medium. Bao et al. \cite{GHPL19} considered the model of scattering by an unbounded rough interface with an impenetrable obstacle, while the background medium is also supposed to be lossy. We further note that all these work were only considering the case of the transmission coefficient $\mu=1$. In other words, the solutions and their normal derivatives are assumed to be continuous across the penetrable interface, which would make the analysis much easier. In \cite{DTS03}, Natroshvili, Arens and Chandler-Wilde proved the well-posedness of a two-dimensional transmission problem with the transmission coefficient $\mu\neq1$ (actually here requires $\I\mu<0$) but only when the medium is lossy. Very few work have undertaken a rigorous theoretical study on the whole space problem with real transmission coefficient $\mu\neq1$ and lossless medium simultaneously, which is significant for both mathmatical and practical interests. Thus in this paper we aim to study the transmission scattering problem formulated in \cite{DTS03} with lossless medium and real transmission coefficient, where the cases $\mu=1$ and $\mu\neq1$ are both considered.  

In the present work, we obtain the existence of the unique solution for the aforementioned problem applying the theory developed by Chandler-Wilde and his collaborators in \cite{SNB98,BS98,SB99,DTS03,BS03,SC96,SB98,SCB99,SB97,SB02,SBC00}. The uniqueness for the forward problem is first proved under some a priori assumptions on the transmission coefficient and the wave numbers, which, to the best of the author's knowledge, is the first uniqueness result in the direct rough surface scattering for $\mu\neq1$ in a lossless medium. The scattering problem is then reduced to an equivalent boundary integral equation system, where the existence of the solution follows from certain properties of the relevant integral operators derived in \cite{DTS03}. Moreover, it shoule be emphasized that our work on the direct problem is not a trivial generalization to \cite{DTS03}, since the conditions $\I\mu<0$ and the lossy medium in \cite{DTS03} do simplify the mathmatical analysis a lot, and a more comprehensive consideration is needed here. After establishing the well-posedness, we turn to further study the property of the solutions to the transmission scattering problem with point source and hypersingular point source incident waves, which is well known that plays a key role in the inverse scattering problem (see, e.g., \cite{VI08,VI90,AR93}). It is a natural thought that the solutions have a same singularity as the incident point sources (or hypersingular point sources) with the source positions approaching the interface. Here we give the first rigorous proof for such thought. We propose a simple and novel perspective to show that the solutions restricted to the neighborhood of the source position and the interface do possess the same singularity as the singular incidence in the $H^1$ sense with a coefficient relative to the transmission coefficent $\mu$. Our method mainly based on the detailed analysis on the integral equations and can be applied to many other cases, such as the electromagnetic and elastic scattering whether by bounded regions or unbounded structures. This local property of the solutions is an important corollary to our well-posedness theorem and also plays a fundamental part in the uniqueness proof of the inverse problem. 

The second part of this paper is devoted to the inverse problem for uniquely determining the scatterer. There exist many uniqueness results on inverse scattering problems on periodic structures \cite{BS99,JB11,JB12,GB94,FA97} and local rough surfaces (or cavities) \cite{JJB22,GJP11,HB13,PL12}, which can be viewed as special cases of rough surfaces. However, the methods utilized in these cases usually rely heavily on the a priori knowledge on the surface at infinity and thus can not be extended to the general unbounded non-periodic rough surfaces. Due to the great difficulty, few related work are available on the uniqueness results for the inverse rough surface scattering. We refer to \cite{GL18,SC95,GH12} for the determination of impenetrable rough surfaces. As for penetrable rough surfaces, 
Bao et al. \cite{GHPL19} showed that the rough surface and the obstacle embedded in a lossy medium both can be uniquely identified by the measured fields corresponding to a single point incident wave. Different from the direct problem, the inverse problem becomes much harder when the transmission coefficient $\mu=1$, for which we note that many uniqueness results on inverse problems in the aforementioned work do not contain the case that $\mu=1$. A recent exception is  \cite{JBH18} by Yang et al., where the uniqueness of inverse scattering from bounded penetrable obstacles is proved under both the situations of $\mu=1$ and $\mu\neq1$. Their main idea is to construct a well-posed interior transmission problem in a small domain near the interface using the scattered fields corresponding to the incident point source or hypersingular point source, which motivates us to study the local property of these solutions in the rough surface case and also construct such interior transmission problem. 
With these preparations, we prove that the entire scatterer, i.e., the unbounded penetrable rough surface, the transmission coefficient and the wave number below the surface can be uniquely determined by the measured near field only on a line segment above the rough surface. 

The remaining part of the paper is organized as follows. In the next section, we introduce the mathmatical formulation for the model transmission problem in details, and present some useful notations and inequalities. Based on the previous work on the upward (downward) propagating radiation condition, in section \ref{sec3} we obtain the uniqueness of the direct problem following the standard procedure but with a more elaborate analysis. Then in section \ref{sec4} we further investigate on the existence of the solution applying the boundary integral equations method. Section \ref{sec5} is devoted to derive the detailed singularity of the solutions with the point source or hypersingular point source incident waves, where a simple and novel perspective for the derivation is proposed. Finally, in section \ref{sec6} we prove a global uniqueness result for the inverse problem. We show that the whole scatterer system can be uniquely recovered from the knowledge of the near field at a fixed frequency.

\section{The model problem}\label{sec2}
\setcounter{equation}{0}
In this section, we present the mathmatical model of the forward transmission scattering problem, and introduce some notations and estimates used thoughout the paper.  

Denote by $B_r(x)$ the open disk centered at $x\in\R^2$ with radius $r>0$. For $h\in\R$, define $U_h^\pm:=\{x=(x_1,x_2)\in\R^2|x_2\gtrless h\}$ and $\Gamma_h:=\{x\in\R^2|x_2=h\}$. For $V\subset\R^n(n=1,2)$, denote by $BC(V)$ the set of functions bounded and continuous in $V$, which is a Banach space under the norm $\|\psi\|_{\infty,V}:=\sup_{x\in V}|\psi(x)|$. For simplicity, we abbreviate $\|\cdot\|_{\infty,\R^n}$ by $\|\cdot\|_{\infty}$. For $0<\alpha\leq1$, denote by $BC^{0,\alpha}(V)$ the Banach space of functions $\varphi\in BC(V)$ which are uniformly H\"{o}lder continuous with exponent $\alpha$, equipped with the norm 
$$\|\varphi\|_{0,\alpha,V}:=\|\varphi\|_{\infty,V}+\sup\limits_{x,y\in V,x\neq  y}\frac{|\varphi(x)-\varphi(y)|}{|x-y|^\alpha}.$$
For an open set $V\subset\R^2$ and $v\in L^\infty(V)$, denote by $\pa_jv$, $j=1,2$, the (distributional) derivative $\pa v(x)/\pa x_j$. Let $BC^1(V):=\{\varphi\in BC(V)|\pa_j\varphi\in BC(V),j=1,2\}$ under the norm $\|\varphi\|_{1,V}:=\|\varphi\|_{\infty,V}+\|\pa_1\varphi\|_{\infty,V}+\|\pa_2\varphi\|_{\infty,V}$. We further introduce some spaces of functions on the boundary $\Gamma$. Define $BC^{1,\alpha}(\Gamma):=\{\varphi\in BC(\Gamma)|\Grad\varphi\in BC^{0,\alpha}(\Gamma)\}$ with the norm $\|\varphi\|_{1,\alpha,\Gamma}:=\|\varphi\|_{\infty,\Gamma}+\|\Grad\varphi\|_{0,\alpha,\Gamma}$, where Grad is the surface gradient. 

Moreover, for a bounded domain $\Om\subset\R^2$ with a $C^2$ boundary, we introduce the Hilbert spaces  $H^1_\Delta(\Om)$ defined by 
\ben
H^1_\Delta(\Om):=\{u\in\mathcal{D}'(\Om)|u\in H^1(\Om),~\Delta u\in L^2(\Om)\} 
\enn
with the inner product 
\ben
(u,v)_{H^1_\Delta(\Om)}:=(u,v)_{H^1(\Om)}+(\Delta u,\Delta v)_{L^2(\Om)},~\forall u,v\in H^1_\Delta(\Om), 
\enn
where $\mathcal{D}'(\Om)$ denotes the set of distributions defined on $C_0^\infty(\Om)$. 

Given $f\in BC^{1,\alpha}(\R)$, $0<\alpha\leq1$, with $f^-:=\inf_{x_1\in\R}f(x_1)$ and $f^+:=\sup_{x_1\in\R}f(x_1)$, define $D^\pm$ by 
\ben
  &&D^+:=\{x=(x_1,x_2)\in\R^2|x_2>f(x_1)\},\\
  &&D^-:=\{x=(x_1,x_2)\in\R^2|x_2<f(x_1)\}, 
\enn
so that the interface $\Gamma$ is 
$$\pa D^+=\pa D^-=\Gamma:=\{x=(x_1,f(x_1))|x_1\in\R\}.$$
When we want to show the dependence of the regions and interface on the function $f$, we write $D^{\pm}_f$ and $\Gamma_f$. Denote by $n(x)=(n_1(x),n_2(x))$ the unit normal vector to $\Gamma$ at $x\in\Gamma$ pointing out of $D^+$, and $\pa_n=n_1\pa_1+n_2\pa_2$ and $\pa_\tau=n_2\pa_1-n_1\pa_2$ the usual normal and tangent derivatives on $\Gamma$. 

Now we formulate the interface problem which models the scattering of acoustic waves by the penetrable unbounded obstacle $D^-$. The incident plane wave $u^i(x)=e^{ik_+x\cdot d}$, $x\in\R^2$, with $d=(d_1,d_2)\in\Sp^1$ the propagation direction, will produce a scattered wave $u_+$ in $D^+$ and a transmitted wave $u_-$ in $D^-$. Note that one could also consider other types of incident waves, e.g., the so-called point source or hypersingular point source waves. The waves $u_\pm$ satisfy the Helmholtz equations in $D^\pm$, and the transmission boundary conditions on $\Gamma$. To be specific, the forward transmission scattering problem associated with the scatterer $(\Gamma,k_-,k_+,\mu)$ is modeled by 
\be\label{2.1}
\left\{
\begin{array}{ll}
	\Delta u_++k_+^2u_+=0~~~&{\rm in}~D^+,\\[2mm]
	\Delta u_-+k_-^2u_-=0~~~&{\rm in}~D^-,\\[2mm]
	u_++u^i=u_-~~~&{\rm on}~\Gamma,\\[2mm]
	\displaystyle\frac{\pa u_+}{\pa n}+\frac{\pa u^i}{\pa n}=\mu\frac{\pa u_-}{\pa n}~~~&{\rm on}~\Gamma, 
\end{array}
\right.
\en
where $k_\pm>0$ are the wave numbers and $\mu>0$ is the transmission coefficient. The direct problem is to determine the scattered and transmitted fields, while the inverse problem considered here is given $k_+$, determining the scatterer $(\Gamma,k_-,k_+,\mu)$ from the knowledge of the scattered wave $u_+$. 

In order to guarantee the uniqueness of problem (\ref{2.1}), the scattered and transmitted waves have to satisfy additional conditions at infinity, which is the so-called upward (downward) propagating radiation condition - UPRC (DPRC). To this end, we introduce the following definitions. 

For $k>0$, denote by  $\Phi_{k}(x,y)$ the fundamental solution to the Helmholtz equation in $\R^2$, which is known as $$\Phi_{k}(x,y):=\frac{i}{4}H_0^{(1)}(k|x-y|), ~~~x\neq y, $$ with $H_0^{(1)}$ the Hankel function of the first kind of order zero. We further denote by $\Phi_0(x,y):=(1/2\pi)\ln(1/|x-y|)$ the fundamental solution to the Laplacian in $\R^2$. 
\begin{definition}\label{def2.1}
	Given a domain $G\subset\R^2$, call $v\in C^2(G)\cap L^\infty(G)$ a radiating solution of the Helmholtz equation in $G$ if $\Delta v+k^2v=0$ in $G$ and 
	\be\label{2.2}
	v(x)=O(r^{-1/2}),~~~\frac{\pa v(x)}{\pa r}-ikv(x)=o(r^{-1/2}),~r=|x|, 
	\en
	as $r$ tends to infinity, uniformly in $x/|x|$. 
\end{definition}

The conditions (\ref{2.2}) are just the classical Sommerfeld radiation condition. The set of radiating functions corresponding to the domain $G$ and the parameter $k$ is denoted by Som$(G,k)$. 

\begin{definition}{\rm(\cite{SB98,DTS03})}\label{def2.2}
	Given a domain $G\subset\R^2$, say that $v:G\rightarrow\C$, a solution of the Helmholtz equation $\Delta v+k^2v=0$ in $G$, satisfies the UPRC (DPRC) in $G$ if, for some $h\in\R$ and $\varphi\in L^\infty(\Gamma_h)$, it holds that $U_h^+\subset G$ $(U_h^-\subset G)$ and 
	\be\label{2.3}
	  v(x)=2\theta\int_{\Gamma_h}\frac{\pa\Phi_k(x,y)}{\pa y_2}\varphi(y)ds(y),~x\in U_h^+~(x\in U_h^-), 
	\en
	where $\theta=1$ for the UPRC and $\theta=-1$ for the DPRC. 
\end{definition}

Denote by UPRC$(G,k)$ (DPRC$(G,k)$)  the set of functions satisfying the UPRC (DPRC) in $G$ with the parameter $k$. 

\begin{remark}\label{remark2.3}
	For $x=(x_1,x_2)\in\R^2$ let $x'=(x_1,-x_2)$, and for $G\subset\R^2$ let $G'=\{x'|x\in G\}$. Then $v_-\in$DPRC$(G,k)$ if and only if $v_+\in$UPRC$(G,k)$, provided $v_+(x)=v_-(x')$, $x\in G'$. 
\end{remark}

 Note that the existence of the integral (\ref{2.3}) for arbitatry $\varphi\in L^\infty(\Gamma_h)$ is assured by the bound which follows from asymptotic behavior of the Hankel function for small and large real argument 
\be\label{2.4}
  \left|\frac{\pa\Phi_k(x,y)}{\pa y_2}\right|\leq C|x_2-y_2|\left(|x-y|^{-2}+|x-y|^{-3/2}\right),~x,y\in\R^2,~x\neq y, 
\en
where $C$ is a positive constant depending only on $k$. We further give the following result which states the properties of the UPRC (DPRC). 
\begin{lemma}{\rm(\cite{SB98})}\label{lem2.4}
	Given $H\in\R$ and $v:U_H^+\rightarrow\C$, the following statements are equivalent: 
	\begin{enumerate}[{\em (i)}]
		\item $v\in C^2(U_H^+)$, $v\in L^\infty(U_H^+\setminus U_a^+)$ for all $a>H$, $\Delta v+k^2v=0$ in $U_H^+$ and $v$ satisfies the UPRC in $U_H^+$; 
		\item there exists a sequence $\{v_n\}$ of radiaiting solutions such that $v_n(x)\rightarrow v(x)$ uniformly on compact subsets of $U_H^+$ and $$\sup\limits_{x\in U_H^+\setminus U_a^+, n\in\N}|v_n(x)|<+\infty$$ for all $a>H$; 
		\item $v$ satisfies {\rm (\ref{2.3})} for $h=H$ and some $\varphi\in L^\infty(\Gamma_h)$; 
		\item $v\in L^\infty(U_H^+\setminus U_a^+)$ for some $a>H$ and $v$ satisfies {\rm (\ref{2.3})} for each $h>H$ with $\varphi=v|_{\Gamma_h}$;
		\item $v\in C^2(U_H^+)$, $v\in L^\infty(U_H^+\setminus U_a^+)$ for all $a>H$, $\Delta v+k^2v=0$ in $U_H^+$, and, for every $h>H$ and radiating solution in $U_H^+$, $w$, such that the restriction of $w$ and $\pa_2 w$ to $\Gamma_h$ are in $L^1(\Gamma_h)$, it holds that $$\int_{\Gamma_h}\left(v\frac{\pa w}{\pa n}-w\frac{\pa v}{\pa n}\right)ds=0. $$
	\end{enumerate}
\end{lemma}
\begin{remark}\label{remark2.5}
	The above lemma implies that any radiating solution satisfies the UPRC (DPRC), thus this radiation condition is in some extents the generalization of the Sommerfeld radiation condition. 
\end{remark}

Now we formulate the previous transmission scattering problem (\ref{2.1}) as follows (see, e.g., \cite{DTS03}). 

{\bf Transmission Scattering Problem} (TSP). Let $\alpha\in(0,1)$. Given $g_1\in BC^{1,\alpha}(\Gamma)$ and $g_2\in BC^{0,\alpha}(\Gamma)$, determine a pair of functions $(u_+,u_-)$ with $u_+\in C^2(D^+)\cap BC^1(\overline{D^+}\setminus U_{h_+}^+)$ $(h_+>f^+)$ and $u_-\in C^2(D^-)\cap BC^1(\ov{D^-}\setminus U_{h_-}^-)$ $(h_-<f^-)$ such that the following hold: 
\begin{enumerate}[(i)]
	\item $u_\pm$ is the solution of the Helmholtz equation $\Delta u_\pm+k_\pm^2u_\pm=0$ in $D^\pm$; 
	\item $u_+-u_-=g_1$, $\pa_n u_+-\mu\pa_n u_-=g_2$ on $\Gamma$;
	\item $\sup\limits_{x\in D^\pm}|x_2|^\beta|u_\pm(x)|<\infty$ for some $\beta\in\R$;
	\item $u_+\in$UPRC$(D^+,k_+)$ and $u_-\in$DPRC$(D^-,k_-)$. 
\end{enumerate}

Note that the TSP has been studied in {\rm \cite{DTS03}} for the case of $k_+>0$, $\Rt k_->0$, $\I k_->0$, $\Rt\mu>0$, $\I\mu<0$ and $\I(\mu k_-^2)>0$. In this paper, we aim to further establish the well-posedness of the TSP for real $k_\pm$ and $\mu$.

To this end, we first introduce some basic notations and inequalities. By the asymptotic behavior of the Hankel function, we know that 
\begin{align}\label{2.5}
  |\Phi_k(x,y)|\;&\leq C|x-y|^{-1/2},\\ \label{2.6}
  \left|\frac{\pa\Phi_k(x,y)}{\pa x_i}\right|\;&\leq C|x_i-y_i||x-y|^{-3/2},~~i=1,2, 
\end{align}
for $|x-y|\geq\delta>0$ with $C>0$ depending only on $k,\delta$. 

Let $x,y\in U_a^\pm$, $a\in\R$ and $y_a'=(y_1,2a-y_2)$. Denote by $G_k^{\pm(\mathcal{D})}(x,y;a)$ and $G_k^{\pm\mathcal{(I)}}(x,y;a)$ the Dirichlet Green's function and the impedance Green's function for the Helmholtz operator $\Delta+k^2$ in the half planes $U_a^\pm$. It is well known that (c.f., \cite{SCB99,BS03,SB98}) 
\begin{align}\label{2.7}
  G_k^{\pm(\mathcal{D})}(x,y;a)&:=\Phi_{k}(x,y)-\Phi_{k}(x,y_a'),~~x,y\in U_a^\pm, \\ \label{2.8}
  G_k^{\pm(\mathcal{I})}(x,y;a)&:=\Phi_k(x,y)+\Phi_k(x,y_a')+P_k^\pm(x-y_a'),~~x,y\in U_a^\pm, 
\end{align}
where 
$$P_k^\pm(z):=\frac{|z|e^{ik|z|}}{\pi}\int_{0}^{\infty}\frac{t^{-1/2}e^{-k|z|t}[|z|\pm z_2(1+it)]}{\sqrt{t-2i}[|z|t-i(|z|\pm z_2)]^2}dt,~~z\in\ov U_0^\pm, $$
with the square root being taken so that $-\pi/2<\arg\sqrt{t-2i}<0$. By \cite{SCB99}, we know that $P_k^\pm\in C(\ov U_0^\pm)\cap C^\infty(\ov U_0^\pm\setminus\{0\})$. Moreover,  $G_k^{\pm(\mathcal{D})}(x,y;a)$ and $G_k^{\pm\mathcal{(I)}}(x,y;a)$ are radiating in $U_a^\pm$ and 
\be\label{2.9}
&&G_k^{\pm(\mathcal{D})}(x,y;a)=0~~\text{for}~~x\in\Gamma_a, \\ \label{2.10}
&&\frac{\pa}{\pa x_2}G_k^{\pm\mathcal{(I)}}(x,y;a)\pm ikG_k^{\pm\mathcal{(I)}}(x,y;a)=0~~\text{for}~~x\in\Gamma_a. 
\en
Furthermore, for $G(x,y)\in\{G_k^{\pm(\mathcal{D})}(x,y;a), G_k^{\pm\mathcal{(I)}}(x,y;a)\}$ there hold the following inequalities (for details see \cite{SB98, SCB99,SC96}) 
\be\label{2.11}
  \begin{aligned}
  	&|G(x,y)|,~|\na_xG(x,y)|,~|\na_yG(x,y)|\\
	&\quad\leq C\frac{(1+|x_2|)(1+|y_2|)}{|x-y|^{3/2}}~~\text{for}~~|x-y|\geq1, \\ 
	&|G(x,y)|\leq C(1+|\log|x-y||)~~\text{for}~~0<|x-y|\leq1, \\
	&|\na_xG(x,y)|,~|\na_yG(x,y)|\leq C|x-y|^{-1}~~\text{for}~~0<|x-y|\leq1, \\
 	&|G(x,y)|,~|\na_x G(x,y)|,~|\na_y G(x,y)|, |\na_x\pa_{n(y)}G(x,y)|\\
  	&\quad\leq C_1[1+|x_1-y_1|]^{-3/2}~~\text{for}~~x\in\Gamma_H,y\in\Gamma, |x_2-y_2|\geq\delta>0, 
  \end{aligned}
\en
where $C>0$ depending on $a$ and $k$, and $C_1>0$ depending on $a$, $k$, $\delta$, $\Gamma$ and $H$. 

Denote by $G^{(\pm)}(x,y)$, the generalized Dirichlet Green's functions for the domains $D^\pm$: 
\ben
  &&G^{(-)}(x,y)=G_{k_-}^{-(\mathcal{D})}(x,y;h_+)-V^{(-)}(x,y),~~y,x\in\ov{D^-}, \\
  &&G^{(+)}(x,y)=G_{k_+}^{+(\mathcal{D})}(x,y;h_-)-V^{(+)}(x,y),~~y,x\in\ov{D^+}, 
\enn
where $V^{(\pm)}(x,y)$ is a solution of the Helmholtz equation $(\Delta+k_\pm^2)u=0$ in $D^\pm$ satisfying the UPRC (DPRC) and the boundary condition 
\ben
  &&V^{(-)}(x,y)=G_{k_-}^{-(\mathcal{D})}(x,y;h_+),~~y\in D^-,~x\in\Gamma, \\
  &&V^{(+)}(x,y)=G_{k_+}^{+(\mathcal{D})}(x,y;h_-),~~y\in D^+,~x\in\Gamma. 
\enn 
Due to the results in \cite{SCB99,SB98}, the functions $V^{(\pm)}(x,y)$ and $G^{(\pm)}(x,y)$ are uniquely determined and radiating, and admit some bounds similar to (\ref{2.11}) (see \cite{DS03})
\be\label{2.12}
\begin{aligned}
	&|G^{(\pm)}(x,y)|,~|\na_xG^{(\pm)}(x,y)|,~|\na_yG^{(\pm)}(x,y)|\\
	&\quad\leq C_\pm^*\frac{(1+|x_2|)(1+|y_2|)}{|x-y|^{3/2}}~~\text{for}~~|x-y|\geq1, \\ 
	&|G^{(\pm)}(x,y)|\leq C_\pm^*(1+|\log|x-y||)~~\text{for}~~0<|x-y|\leq1, \\
	&|\na_xG^{(\pm)}(x,y)|,~|\na_yG^{(\pm)}(x,y)|\leq C_\pm^*|x-y|^{-1}~~\text{for}~~0<|x-y|\leq1, \\
	&|G^{(\pm)}(x,y)|,~|\na_x G^{(\pm)}(x,y)|,~|\na_y G^{(\pm)}(x,y)|, |\na_x\pa_{n(y)}G^{(\pm)}(x,y)|\\
	&\quad\leq C_\pm^{**}[1+|x_1-y_1|]^{-3/2}~~\text{for}~~x\in\Gamma_{a_\pm},y\in\Gamma, a_-<f^-<f^+<a_+, 
\end{aligned}
\en
where $C_\pm^*>0$ depending on $h_\pm$ and $k_\pm$, and $C_\pm^{**}>0$ depending on $h_\pm$, $k_\pm$, $a_\pm$, $\Gamma$. 

For later use, we here give an integral representation  and a local regularity estimate for the solutions. 
\begin{lemma}{\rm (\cite[Lemma 2.4]{DTS03})}\label{lem2.6}
	Let $u_\pm\in C^2(D^\pm)\cap C^1(\ov{D^\pm})$ be a solution to the equation $(\Delta+k_\pm^2)u_\pm=0$ in $D^\pm$ with $u_+$ and $u_-$ satsifying the UPRC and DPRC, respectively. Then $$u_\pm(x)=\pm\int_{\Gamma}\frac{\pa G^{(\pm)}(x,y)}{\pa n(y)}u_\pm(y)ds(y),~~x\in D^\pm. $$
\end{lemma}
\begin{lemma}{\rm (\cite[Theorem 3.9 and Lemma 4.1]{DN98})}\label{lem2.7}
	If $G\subset\R^2$ is open and bounded, $v\in L^\infty(G)$, and $\Delta v=f\in L^\infty(G)$ (in a distributional sense) then $v\in C^1(G)$ and 
	$$|\na v(x)|\leq C(d(x))^{-1}(\|v\|_{\infty,G}+\|d^2f\|_{\infty,G}),~~~x\in G, $$
	where $C$ is a positive constant and $d(x)=d(x,\pa G)$. 
\end{lemma}

\section{Uniqueness of solution}\label{sec3}
\setcounter{equation}{0}
In this section, we will establish the uniqueness theorem for the TSP with real coefficients $k_\pm,\mu>0$. To the best of author's knowledge, this is the first uniqueness result covering both the cases that $\mu\neq1$ and the medium is lossless. 
  
Before going further, we introduce some notations used in the remaining part of the paper. For $A>0$, $h_-<f^-$ and $h_+>f^+$ define 
\ben
  &&D^\pm(A):=\{x\in D^\pm||x_1|<A\},~~~\Gamma_h(A):=\{x\in\Gamma_h||x_1|<A\}, \\
  &&\Gamma(A):=\{x\in\Gamma||x_1|<A\},~~~D^-_{h_-}:=D^-\setminus\ov{U_{h_-}^-}=U_{h_-}^+\setminus\ov{D^+}, \\
  &&D^+_{h_+}:=D^+\setminus\ov{U_{h_+}^+}=U_{h_+}^-\setminus\ov{D^-},~~~D^\pm_{h_\pm}(A):=\{x\in D^\pm_{h_\pm}||x_1|<A\}, \\
  &&\gamma_+(\pm A):=\{x\in D^+_{h_+}|x_1=\pm A\},~~~\gamma_-(\pm A):=\{x\in D^-_{h_-}|x_1=\pm A\}. 
\enn
In the uniqueness proof we also utilize the following two lemmas. 
\begin{lemma}{\rm (\cite[Lemma 5.2]{SB99})}\label{lem3.1}
	If $\phi_\pm\in L^2(\Gamma_h)\cap L^\infty(\Gamma_h)$, and $v_\pm$ are defined by {\rm (\ref{2.3})} with the densities $\phi_\pm$ and $v_-,v_+$ satisfying the DPRC, UPRC, respectively. Then the restrictions of $v_-,\pa_1v_-$ and $\pa_2v_-$ to $\Gamma_a$ are in $L^2(\Gamma_a)\cap BC(\Gamma_a)$ for $a<h$; the restrictions of $v_+,\pa_1v_+$ and $\pa_2v_+$ to $\Gamma_b$ are in $L^2(\Gamma_b)\cap BC(\Gamma_b)$ for $b>h$; and 
	\begin{align}\label{3.1}
	  &\I\int_{\Gamma_a}\ov v_-\pa_2v_-ds\leq0,~~~~\Rt\int_{\Gamma_a}\ov v_-\pa_2v_-ds\geq0, \\ \label{3.2}
	  &\int_{\Gamma_a}[|\pa_2v_-|^2-|\pa_1v_-|^2+k^2|v_-|^2]ds\leq-2k\I\int_{\Gamma_a}\ov v_-\pa_2v_-ds, 
	\end{align}
	and 
	\begin{align}\label{3.3}
	&\I\int_{\Gamma_b}\ov v_+\pa_2v_+ds\geq0,~~~~\Rt\int_{\Gamma_b}\ov v_+\pa_2v_+ds\leq0, \\ \label{3.4}
	&\int_{\Gamma_b}[|\pa_2v_+|^2-|\pa_1v_+|^2+k^2|v_+|^2]ds\leq2k\I\int_{\Gamma_b}\ov v_+\pa_2v_+ds. 
	\end{align}
\end{lemma} 


\begin{lemma}{\rm (\cite[Lemma 6.2]{SNB98})}\label{lem3.2}
	Suppose that $F\in L^\infty(\R)$ and that, for some non-negative constants $C,\varepsilon$ and $A_0$, 
	$$\int_{-A}^{A}|F|^2\leq C\int_{\R\setminus[-A,A]}G_A^2+C\int_{-A}^{A}(G_\infty-G_A)G_\infty+\varepsilon,~~~A>A_0, $$
	where, for $A_0<A\leq+\infty$, 
	$$G_A(s)=\int_{-A}^{A}(1+|s-t|)^{-3/2}|F(t)|dt,~~s\in\R. $$
	Then $F\in L^2(\R)$ and 
	$$\int_{-\infty}^{+\infty}|F|^2\leq\varepsilon. $$
\end{lemma}
\begin{theorem}\label{thm3.3}
	Suppose $(u_+,u_-)$ is the solution to the homogeneous TSP, i.e., $g_1=g_2=0$, then $u_\pm=0$ in $D^\pm$, provided $(\Gamma,k_-,k_+,\mu)$ satisfying $(\mu-1)(k_+^2-k_-^2\mu)\geq0$ and $k_+^2\neq k_-^2\mu$. 
\end{theorem}
\begin{proof}
	Applying Green's first theorem to $u_\pm$ and $\ov u_\pm$ in $D^\pm_{h_\pm}(A)$ yields
	\begin{align}\label{3.5}\nonumber
	  \int_{D^-_{h_-}(A)}\left(|\na u_-|^2-k_-^2|u_-|^2\right)dx\;&=-\int_{\Gamma(A)}\ov u_-\pa_nu_-ds-\int_{\Gamma_{h_-}(A)}\ov u_-\pa_2u_-ds\\
	  &\quad+\left(\int_{\gamma_-(A)}-\int_{\gamma_-(-A)}\right)\ov u_-\pa_1u_-ds, \\\nonumber\label{3.6}
	  \int_{D^+_{h_+}(A)}\left(|\na u_+|^2-k_+^2|u_+|^2\right)dx\;&=\int_{\Gamma(A)}\ov u_+\pa_nu_+ds+\int_{\Gamma_{h_+}(A)}\ov u_+\pa_2u_+ds\\
	  &\quad+\left(\int_{\gamma_+(A)}-\int_{\gamma_+(-A)}\right)\ov u_+\pa_1u_+ds. 
	\end{align}
	Multiplying (\ref{3.5}) by $\mu$, adding to (\ref{3.6}), using the transmission boundary conditions and taking the imaginary part of the equation, we obtain that 
	\begin{align}\label{3.7}\nonumber
	  &\;\quad\mu\I\int_{\Gamma_{h_-}(A)}\ov u_-\pa_2u_-ds-\I\int_{\Gamma_{h_+}(A)}\ov u_+\pa_2u_+ds \\
	  &=\mu\I\left(\int_{\gamma_-(A)}-\int_{\gamma_-(-A)}\right)\ov u_-\pa_1u_-ds+\I\left(\int_{\gamma_+(A)}-\int_{\gamma_+(-A)}\right)\ov u_+\pa_1u_+ds.~ 
	\end{align}
	
	Let $\theta\in C^2[h_-,h_+]$ be a real valued function to be decided. Apply Green's first theorem to $u_\pm$ and $\theta(x_2)\ov u_\pm$ in $D^\pm_{h_\pm}(A)$ and take the real part of the results to derive 
	\begin{align}\label{3.8}\nonumber
	  &\quad-\int_{\Gamma(A)}\theta'(x_2)n_2|u_-|^2ds+2\Rt\int_{\Gamma(A)}\theta(x_2)\ov u_-\pa_nu_-ds \\ \nonumber
	  &=-\int_{D^-_{h_-}(A)}2\theta(x_2)|\na u_-|^2dx+\int_{D^-_{h_-}(A)}(2k_-^2\theta(x_2)+\theta''(x_2))|u_-|^2dx \\ \nonumber
	  &\quad-2\theta(h_-)\Rt\int_{\Gamma_{h_-}(A)}\ov u_-\pa_2u_-ds+2\Rt\left(\int_{\gamma_-(A)}-\int_{\gamma_-(-A)}\right)\theta(x_2)\ov u_-\pa_1u_-ds \\
	  &\quad+\theta'(h_-)\int_{\Gamma_{h_-}(A)}|u_-|^2ds, 
	\end{align}
	and
	\begin{align}\label{3.9}\nonumber
	  &\quad\int_{\Gamma(A)}\theta'(x_2)n_2|u_+|^2ds-2\Rt\int_{\Gamma(A)}\theta(x_2)\ov u_+\pa_nu_+ds \\ \nonumber
	  &=-\int_{D^+_{h_+}(A)}2\theta(x_2)|\na u_+|^2dx+\int_{D^+_{h_+}(A)}(2k_+^2\theta(x_2)+\theta''(x_2))|u_+|^2dx \\ \nonumber
	  &\quad+2\theta(h_+)\Rt\int_{\Gamma_{h_+}(A)}\ov u_+\pa_2u_+ds+2\Rt\left(\int_{\gamma_+(A)}-\int_{\gamma_+(-A)}\right)\theta(x_2)\ov u_+\pa_1u_+ds \\
	  &\quad-\theta'(h_+)\int_{\Gamma_{h_+}(A)}|u_+|^2ds. 
	\end{align}
	Then it follows from the boundary conditions that 
	\begin{align}\label{3.10}\nonumber
	  &\quad-\int_{\Gamma(A)}(\mu-1)\theta'(x_2)n_2|u_-|^2ds \\ \nonumber
	  &=-\mu\int_{D^-_{h_-}(A)}2\theta(x_2)|\na u_-|^2dx+\mu\int_{D^-_{h_-}(A)}(2k_-^2\theta(x_2)+\theta''(x_2))|u_-|^2dx \\\nonumber
	  &\quad-\int_{D^+_{h_+}(A)}2\theta(x_2)|\na u_+|^2dx+\int_{D^+_{h_+}(A)}(2k_+^2\theta(x_2)+\theta''(x_2))|u_+|^2dx \\ \nonumber
	  &\quad-2\mu\theta(h_-)\Rt\int_{\Gamma_{h_-}(A)}\ov u_-\pa_2u_-ds+2\mu\Rt\left(\int_{\gamma_-(A)}-\int_{\gamma_-(-A)}\right)\theta(x_2)\ov u_-\pa_1u_-ds \\ \nonumber
	  &\quad+2\theta(h_+)\Rt\int_{\Gamma_{h_+}(A)}\ov u_+\pa_2u_+ds+2\Rt\left(\int_{\gamma_+(A)}-\int_{\gamma_+(-A)}\right)\theta(x_2)\ov u_+\pa_1u_+ds \\ 
	  &\quad+\mu\theta'(h_-)\int_{\Gamma_{h_-}(A)}|u_-|^2ds-\theta'(h_+)\int_{\Gamma_{h_+}(A)}|u_+|^2ds. 
	\end{align}
	
	Denote by $\varphi\in C^1[h_-,h_+]$ a real valued function at our disposal. Multiplying the equations in (\ref{2.1}) by $2\varphi(x_2)\pa_2\ov u_\pm$, integrating over $D^\pm_{h_\pm}(A)$, and taking the real part, we yield that 
	\begin{align}\label{3.11}\nonumber
	  &\quad\int_{\Gamma(A)}\varphi(x_2)[2\Rt(\pa_nu_-\pa_2\ov u_-)-n_2|\na u_-|^2+k_-^2n_2|u_-|^2]ds\\ \nonumber
	  &=-\varphi(h_-)\int_{\Gamma_{h_-}(A)}(|\pa_2u_-|^2-|\pa_1u_-|^2+k_-^2|u_-|^2)ds-\int_{D^-_{h_-}(A)}\varphi'(x_2)|\pa_2u_-|^2dx \\ \nonumber
	  &\quad+\Rt\left(\int_{\gamma_-(A)}-\int_{\gamma_-(-A)}\right)2\varphi(x_2)\pa_1u_-\pa_2\ov u_-ds+\int_{D^-_{h_-}(A)}\varphi'(x_2)|\pa_1 u_-|^2dx \\
	  &\quad-\int_{D^-_{h_-}(A)}k_-^2\varphi'(x_2)|u_-|^2dx, 
	\end{align}
	and 
	\begin{align}\label{3.12}\nonumber
	&\quad-\int_{\Gamma(A)}\varphi(x_2)[2\Rt(\pa_nu_+\pa_2\ov u_+)-n_2|\na u_+|^2+k_+^2n_2|u_+|^2]ds\\ \nonumber
	&=\varphi(h_+)\int_{\Gamma_{h_+}(A)}(|\pa_2u_+|^2-|\pa_1u_+|^2+k_+^2|u_+|^2)ds-\int_{D^+_{h_+}(A)}\varphi'(x_2)|\pa_2u_+|^2dx \\ \nonumber
	&\quad+\Rt\left(\int_{\gamma_+(A)}-\int_{\gamma_+(-A)}\right)2\varphi(x_2)\pa_1u_+\pa_2\ov u_+ds+\int_{D^+_{h_+}(A)}\varphi'(x_2)|\pa_1 u_+|^2dx \\
	&\quad-\int_{D^+_{h_+}(A)}k_+^2\varphi'(x_2)|u_+|^2dx. 
	\end{align}
	Since $u_+=u_-$, $\pa_nu_+=\mu\pa_nu_-$ on $\Gamma$, we have $\pa_2(u_+-u_-)=n_2\pa_n(u_+-u_-)$ and $|\na u_+|^2=\mu^2|\pa_n u_-|^2+|\pa_\tau u_-|^2$ on $\Gamma$. Thus we deduce that 
	\begin{align}\label{3.13}\nonumber
      &\quad-\int_{\Gamma(A)}(k_+^2-k_-^2\mu)\varphi(x_2)n_2|u_-|^2ds-\int_{\Gamma(A)}(\mu-1)\varphi(x_2)n_2(\mu|\pa_nu_-|^2+|\pa_\tau u_-|^2)ds \\ \nonumber
      &=-\mu\varphi(h_-)\int_{\Gamma_{h_-}(A)}(|\pa_2u_-|^2-|\pa_1u_-|^2+k_-^2|u_-|^2)ds-\mu\int_{D^-_{h_-}(A)}\varphi'(x_2)|\pa_2u_-|^2dx \\ \nonumber 
      &\quad+\varphi(h_+)\int_{\Gamma_{h_+}(A)}(|\pa_2u_+|^2-|\pa_1u_+|^2+k_+^2|u_+|^2)ds-\int_{D^+_{h_+}(A)}\varphi'(x_2)|\pa_2u_+|^2dx \\ \nonumber
      &\quad+\mu\Rt\left(\int_{\gamma_-(A)}-\int_{\gamma_-(-A)}\right)2\varphi(x_2)\pa_1u_-\pa_2\ov u_-ds+\mu\int_{D^-_{h_-}(A)}\varphi'(x_2)|\pa_1 u_-|^2dx \\ \nonumber
      &\quad+\Rt\left(\int_{\gamma_+(A)}-\int_{\gamma_+(-A)}\right)2\varphi(x_2)\pa_1u_+\pa_2\ov u_+ds+\int_{D^+_{h_+}(A)}\varphi'(x_2)|\pa_1 u_+|^2dx \\ 
      &\quad-\mu\int_{D^-_{h_-}(A)}k_-^2\varphi'(x_2)|u_-|^2dx-\int_{D^+_{h_+}(A)}k_+^2\varphi'(x_2)|u_+|^2dx. 
	\end{align}
	
	For $A>0$, $h_-<f^-$ and $h_+>f^+$, define 
	\ben
	  &&J_{\pm,A}=\I\int_{\Gamma_{h_\pm}(A)}\ov u_\pm\pa_2u_\pm ds,~~~L_{\pm,A}=\Rt\int_{\Gamma_{h_\pm}(A)}\ov u_\pm\pa_2u_\pm ds, \\
	  &&I_{\pm,A}=\int_{\Gamma_{h_\pm}(A)}(|\pa_2u_\pm|^2-|\pa_1u_\pm|^2+k_\pm
	  ^2|u_\pm|^2)ds,~~~K_A=\int_{\Gamma(A)}|u_-|^2ds. 
	\enn
	
	Now we set $\varphi(t)=2(t-h_+')$ if $1\geq\mu>0$, $k_-^2\mu>k_+^2$, and set $\varphi(t)=2(t-h_-')$ if $\mu\geq1$, $k_-^2\mu<k_+^2$, where $h_-<h_-'<f^-$ and $f^+<h_+'<h_+$. We further let $\theta(t)=1$. Also we note that $-n_2=(1+|f'|^2)^{-1/2}\geq\varepsilon$ on $\Gamma$ for some $\varepsilon>0$ since $f\in BC^1(\R)$. Then combinbing (\ref{3.10}) and (\ref{3.13}), by calculation it yields that 
	\be\label{3.14}
	  K_A\leq C_1I_{+,A}+C_2I_{-,A}+C_3L_{+,A}-C_4L_{-,A}+C_5(R_-(A)+R_+(A)), 
	\en
	where 
	\ben
	 R_\pm(A)=\left(\int_{\gamma_\pm(A)}+\int_{\gamma_\pm(-A)}\right)\left(|\pa_1u_\pm\pa_2\ov u_\pm|+|\ov u_\pm\pa_1u_\pm|\right)ds, 
	\enn
	and $C_i$, $i=1,2,3,4,5$ are positive constants depending only on $h_-,h_-',h_+',h_+,\varepsilon,\mu$. 
	
	In view of Lemma \ref{lem2.6}, to utilize Lemma \ref{lem3.2} and the inequality (\ref{3.14}), we define 
	\be\label{3.15}
	  v_\pm(x)=\pm\int_{\Gamma(A)}\frac{\pa G^{(\pm)}(x,y)}{\pa n(y)}u_\pm(y)ds(y),~~x\in D^\pm.
	\en
	Since $\ov{\Gamma(A)}$ is compact, $v_\pm$ is radiating in $D^\pm$. Due to the bounds (\ref{2.12}) (cf. \cite[Lemma 6.1]{SNB98}), we have $v_\pm|_{\Gamma_{h_\pm}}\in L^2(\Gamma_{h_\pm})\cap BC(\Gamma_{h_\pm})$. We further set 
	\ben
	  &&J_{\pm,A}'=\I\int_{\Gamma_{h_\pm}(A)}\ov v_\pm\pa_2v_\pm ds,~~~J_{\pm,A}''=\I\int_{\Gamma_{h_\pm}}\ov v_\pm\pa_2v_\pm ds, \\
	  &&L_{\pm,A}'=\Rt\int_{\Gamma_{h_\pm}(A)}\ov v_\pm\pa_2v_\pm ds,~~~L_{\pm,A}''=\Rt\int_{\Gamma_{h_\pm}}\ov v_\pm\pa_2v_\pm ds, \\
	  &&I_{\pm,A}'=\int_{\Gamma_{h_\pm}(A)}(|\pa_2v_\pm|^2-|\pa_1v_\pm|^2+k_\pm^2|v_\pm|^2)ds, \\
	  &&I_{\pm,A}''=\int_{\Gamma_{h_\pm}}(|\pa_2v_\pm|^2-|\pa_1v_\pm|^2+k_\pm^2|v_\pm|^2)ds. 
	\enn
	Then, by Lemma \ref{lem3.1}, 
	\ben
	  &&J_{-,A}''\leq0,~~~L_{-,A}''\geq0,~~~I_{-,A}''\leq-2k_-J_{-,A}'', \\
	  &&J_{+,A}''\geq0,~~~L_{+,A}''\leq0,~~~I_{+,A}''\leq2k_+J_{+,A}''. 
	\enn
	Hence, from (\ref{3.7}), (\ref{3.14}) and the preceding inequalities we derive that 
	\begin{align}\label{3.16}\nonumber
	  K_A\leq&\;C_1(I_{+,A}-I_{+,A}'')+C_2(I_{-,A}-I_{-,A}'')+C_3(L_{+,A}-L_{+,A}'')+C_4(L_{-,A}''-L_{-,A}) \\
	  &\;+C_1'(J_{+,A}''-J_{+,A})+\mu C_1'(J_{-,A}-J_{-,A}'')+C_5'(R_-(A)+R_+(A)), 
	\end{align}
	where $C_1'=\max\{2C_1k_+,2C_2k_-/\mu\}$ and $C_5'>0$ is a constant depending on $C_5,C_1',\mu$. 
	
	Let 
	\be\label{3.17}
	  w(x_1):=u_-(x_1,f(x_1))=u_-(x)|_{\Gamma}=u_+(x)|_\Gamma. 
	\en
	Clearly, $w\in BC(\R)$ and 
	\be\label{3.18}
	  \int_{-A}^{A}|w(x_1)|^2dx_1\leq K_A\leq C\int_{-A}^A|w(x_1)|^2dx_1 
	\en
	for some positive constant $C$ independent of $A$. Moreover, we define 
	\ben
	  W_A(x_1):=\int_{-A}^{A}(1+|x_1-y_1|)^{-3/2}|w(y_1)|dy_1 
	\enn
	with $0\leq A\leq+\infty$. By (\ref{2.12}) and (\ref{3.15}) we get that for $x\in\Gamma_{h_\pm}$, 
	\begin{align*}
	  |v_\pm(x)|,~~~~|\na v_\pm(x)|&\;\leq CW_A(x_1), \\
	  |u_\pm(x)-v_\pm(x)|,~~~~|\na u_\pm(x)-\na v_\pm(x)|&\;\leq C(W_\infty(x_1)-W_A(x_1)), 
	\end{align*}
    where $C>0$ is a constant independent of $A$. It then follows that 
    \begin{align*}
      |I_{\pm,A}'-I_{\pm,A}''|, |J_{\pm,A}'-J_{\pm,A}''|, |L_{\pm,A}'-L_{\pm,A}''|&\;\leq C\int_{\R\setminus[-A,A]}(W_A(x_1))^2dx_1, \\
      |I_{\pm,A}-I_{\pm,A}'|, |J_{\pm,A}-J_{\pm,A}'|, |L_{\pm,A}-L_{\pm,A}'|&\;\leq C\int_{-A}^A(W_\infty(x_1)-W_A(x_1))W_A(x_1)dx_1, 
    \end{align*}
    so that, from (\ref{3.16}) and (\ref{3.18}) for some constant $C_0>0$ and all $A>0$, 
    \begin{align*}
      \int_{-A}^{A}|w(x_1)|^2dx_1\leq&\;C_0\left\{\int_{-A}^A(W_\infty(x_1)-W_A(x_1))W_A(x_1)dx_1\right. \\
      &\quad\quad\left.+\int_{\R\setminus[-A,A]}(W_A(x_1))^2dx_1+R_-(A)+R_+(A)\right\}. 
    \end{align*}
    Applying Lemma \ref{lem3.2} it is obtained that $w\in L^2(\R)$, i.e., $u_-|_\Gamma,u_+|_\Gamma\in L^2(\Gamma)$ and for all $A_0>0$, 
    \ben
      \int_{\Gamma}|u_\pm|^2ds\leq C_0\sup\limits_{A>A_0}(R_-(A)+R_+(A)),~~j=1,2. 
    \enn
    Following the similar arguments in \cite{DTS03} (see the proof of Theorem 3.1, Step  5-7) and by Lemma \ref{lem2.7}, we get that $R_\pm(A)\rightarrow0$ as $A\rightarrow\infty$. Thus $u_\pm=0$ on $\Gamma$, and hence, by Lemma \ref{lem2.6}, $u_\pm=0$ in $D^\pm$. The proof is finally complete. 
\end{proof}
\begin{remark}\label{remark3.4}
	The conditions on the scatterer $(\Gamma,k_-,k_+,\mu)$ in some extents correspond to the diffraction grating case (see, e.g., the conditions in {\rm \cite[Theorem 3.1]{BS99}, \cite[Theorem 2.2]{JB11} and \cite[Theorem 2.2]{JB12}}). But we here obtain the uniqueness without any a priori assumption of quasi-periodicity. 
\end{remark}

\section{Existence of solution}\label{sec4}
\setcounter{equation}{0}
In this section, we equivalently reduce the TSP to a system of boundary integral equation on $\Gamma$ and show the existence of the unique solution to the TSP by proving the uniqueness result of the corresponding boundary integral equations and combining some existing work. 
	
In this section we always assume that $\Gamma\in BC^{1,1}$ if not otherwise stated. For $a<f^-$ define the single- and double-layer potentials: 
\begin{align*}
(\mathcal{S}_{k_+,a}^+\varphi)(x)&:=\int_{\Gamma}G_{k_+}^{+\mathcal{(I)}}(x,y;a)\varphi(y)ds(y),~~~x\in U_a^+\setminus\Gamma, \\
(\mathcal{D}_{k_+,a}^+\varphi)(x)&:=\int_{\Gamma}\frac{\pa}{\pa n(y)}G_{k_+}^{+\mathcal{(I)}}(x,y;a)\varphi(y)ds(y),~~~x\in U_a^+\setminus\Gamma. 
\end{align*}
And define the boundary integral operators: 
\begin{align*}
(S_{k_+,a}^+\varphi)(x)&:=\int_{\Gamma}G_{k_+}^{+\mathcal{(I)}}(x,y;a)\varphi(y)ds(y),~~~x\in\Gamma, \\
(K_{k_+,a}^+\varphi)(x)&:=\int_{\Gamma}\frac{\pa}{\pa n(y)}G_{k_+}^{+\mathcal{(I)}}(x,y;a)\varphi(y)ds(y),~~~x\in\Gamma, \\
(K_{k_+,a}^{\prime+}\varphi)(x)&:=\int_{\Gamma}\frac{\pa}{\pa n(x)}G_{k_+}^{+\mathcal{(I)}}(x,y;a)\varphi(y)ds(y),~~~x\in\Gamma, \\
(T_{k_+,a}^+\varphi)(x)&:=\frac{\pa}{\pa n(x)}\int_{\Gamma}\frac{\pa}{\pa n(y)}G_{k_+}^{+\mathcal{(I)}}(x,y;a)\varphi(y)ds(y),~~~x\in\Gamma. 
\end{align*}
Further, for $a>f^+$ similarly define the layer potential operators $\mathcal{S}_{k_-,a}^-,\mathcal{D}_{k_-,a}^-$ for $x\in U_a^-\setminus\Gamma$ and the boundary integral operators $S_{k_-,a}^-,K_{k_-,a}^-,K_{k_-,a}^{\prime-},T_{k_-,a}^-$ for $x\in\Gamma$. For detailed properties of these operators, we refer the readers to \cite[Lemma 4.1-4.4]{DTS03}.

We seek for the solution to the TSP in the form of 
\begin{align}\label{4.12}
u_+(x)&=(\mathcal{D}_{k_+,h_-}^+\varphi)(x)+(\mathcal{S}_{k_+,h_-}^+\psi)(x),~~x\in D^+, \\ \label{4.13}
u_-(x)&=\mu^{-1}[(\mathcal{D}_{k_-,h_+}^-\varphi)(x)+(\mathcal{S}_{k_-,h_+}^-\psi)(x)],~~x\in D^-. 
\end{align}
Then by \cite[Lemma 4.1-4.4]{DTS03} we can equivalently reduce the TSP to the system of boundary integral equation on $\Gamma$: 
\be\label{4.17}
  \mathcal{M}\chi=G
\en
where $\chi=(\varphi,\psi)^T$ are unknown density functions, $G=(g_1,g_2)^T$ with 
\be\label{4.16}
g_1\in BC^{1,\alpha}(\Gamma),~g_2\in BC^{0,\alpha}(\Gamma),~ \alpha\in(0,1), 
\en
and
\ben
\mathcal{M}=\left(
\begin{array}{cc}
	K_{k_+,h_-}^+-\mu^{-1}K_{k_-,h_+}^-+(1+\mu)(2\mu)^{-1}I & S_{k_+,h_-}^+-\mu^{-1}S_{k_-,h_+}^-  \\
	T_{k_-,h_+}^--T_{k_+,h_-}^+ & K_{k_-,h_+}^{\prime-}-K_{k_+,h_-}^{\prime+}+I \\
\end{array}
\right) 
\enn
We note that by the known results in \cite[Section 4.3]{DTS03}, to show the existence of solution to the equation (\ref{4.17}), it suffices to prove that it has at most one solution, which is the Theorem \ref{thm4.6} below. 

\begin{lemma}{\rm (\cite[Lemma 4.5]{DTS03})}\label{lem4.5}
    Suppose conditions {\rm (\ref{4.16})} be fulfilled and $\chi=(\varphi,\psi)^T$ be a solution of the equation {\rm (\ref{4.17})}. Then $\varphi\in BC^{1,\alpha}(\Gamma)$ and $\psi\in BC^{0,\alpha}(\Gamma)$ with the same $\alpha$ as in {\rm (\ref{4.16})}. 
\end{lemma}
\begin{theorem}\label{thm4.6}
	If $(\mu-1)(k_+^2-k_-^2\mu)\geq0$ and $k_+^2\neq k_-^2\mu$, then equation {\rm (\ref{4.17})} has at most one solution, i.e., the operator $\mathcal{M}$ is injective. 
\end{theorem}
\begin{proof}
	It suffices to show the homogeneous version of equation (\ref{4.17}) has only the trivial solution. Suppose $\chi=(\varphi,\psi)^T\in[BC(\Gamma)]^2$ solve the equation 
	\be\label{4.18}
	\mathcal{M}\chi=0, 
	\en
	Define $u_-$ and $u_+$ by (\ref{4.12}) and (\ref{4.13}), respectively, with the corresponding $\varphi$ and $\psi$. By Lemma \ref{lem4.5} we know $\varphi\in BC^{1,\beta}(\Gamma)$ and $\psi\in BC^{0,\beta}(\Gamma)$ for all $\beta\in(0,1)$. It then follows from \cite[Lemma 4.1,4.4]{DTS03} and Theorem \ref{thm3.3} that $u_\pm=0$ in $D^\pm$.  
	
	We now further consider the same functions $u_-$ and $u_+$ in the domains $D^+_{h_+}$ and $D^-_{h_-}$, respectively. Applying \cite[Lemma 4.1]{DTS03} we see that for $x\in\Gamma$
	\begin{align}\label{4.19}
	  [u_-(x)]^+-[u_-(x)]^-=\mu^{-1}\varphi(x)&,~[\pa_nu_-(x)]^+-[\pa_nu_-(x)]^-=-\mu^{-1}\psi(x), \\ \label{4.20}
	  [u_+(x)]^+-[u_+(x)]^-=\varphi(x)&,~[\pa_nu_+(x)]^+-[\pa_nu_+(x)]^-=-\psi(x),
	\end{align}
	where $[\cdot]^+$ and $[\cdot]^-$ denote the limits on $\Gamma$ from $D^+$ and $D^-$, respectively. Since $u_\pm=0$ in $D^\pm$, we derive that 
	\be\label{4.21}
	  -\varphi(x)=[u_+(x)]^-=-\mu[u_-(x)]^+,~\psi(x)=[\pa_nu_+(x)]^-=-\mu[\pa_nu_-(x)]^+ 
	\en
	for $x\in\Gamma$. Set 
	\ben
	  &&D^-_*:=D^+_{h_+},~~~D^+_*:=D^-_{h_-}, \\
	  &&v_-(x):=-\mu u_-(x)~~\text{in}~D^-_*,~~~v_+(x):=u_+(x)~~\text{in}~D^+_*. 
	\enn
	It can be verified from (\ref{2.10}) and (\ref{4.21}) that $v_-$ and $v_+$ satisfy the following boundary value problem 
	\be\label{4.22}
	\left\{
	\begin{array}{ll}
		\Delta v_\pm+k_\pm^2v_\pm=0~~~&{\rm in}~D^\pm_*, \\[2mm]
		[v_-]^+=[v_+]^-,~~[\pa_nv_-]^+=[\pa_nv_+]^-~~~&{\rm on}~\Gamma,\\[2mm]
		\pa_2v_-=ik_-v_-~~~&{\rm on}~\Gamma_{h_+}, \\[2mm]
		\pa_2v_+=-ik_+v_+~~~&{\rm on}~\Gamma_{h_-}, 
	\end{array}
	\right.
	\en
	and by \cite[Lemma 4.4]{DTS03} we have $v_\pm\in BC^1(\ov{D^\pm_*})$. 
	
	Define 
	\ben
	  &&D^\pm_*(A,B):=\{x\in D^\pm_*|A<x_1<B\},~~\Gamma(A,B):=\{x\in\Gamma|A<x_1<B\}, \\
	  &&\Gamma_{h_\pm}(A,B):=\{x\in\Gamma_{h_\pm}|A<x_1<B\},~~l_\pm(A):=\{x\in D^\pm_*|x_1=A\}. 
	\enn
	By (\ref{4.22}) and Green's formula, the same as in (\ref{3.5})-(\ref{3.7}), we obtain that 
	\begin{align}\label{4.23}\nonumber
	  &k_-\int_{\Gamma_{h_+}(A,B)}|v_-|^2ds+k_+\int_{\Gamma_{h_-}(A,B)}|v_+|^2ds \\
	  &=\I\left(\widetilde{R}_-(A)-\widetilde{R}_-(B)+\widetilde{R}_+(A)-\widetilde{R}_+(B)\right), 
	\end{align}
	where 
	\ben
	  \wid R_\pm(P):=\int_{l_\pm(P)}\ov v_\pm\pa_1v_\pm ds,~P=A,B. 
	\enn
	Since $v_\pm\in BC^1(\ov{D^\pm_*})$, we have $|\wid R_\pm(P)|$ are uniformly bounded for $P\in(-\infty,+\infty)$. Then by (\ref{4.23}) we deduce that $v_-|_{\Gamma_{h_+}}\in L^2(\Gamma_{h_+})$ and $v_+|_{\Gamma_{h_-}}\in L^2(\Gamma_{h_-})$. Again due to $v_\pm\in BC^1(\ov{D^\pm_*})$, we further yield that 
	\be\label{4.24}
	  &&v_-(x_1,h_+)\rightarrow0~~\text{and}~~v_+(x_1,h_-)\rightarrow0~~\text{as}~|x_1|\rightarrow+\infty. 
	\en
	
	Now integrating by parts in $D^\pm_*$, the same as in (\ref{3.8})-(\ref{3.14}), we derive that 
	\begin{align}\label{4.25}\nonumber
	  \int_{\Gamma(A,B)}|v_-|^2\leq&\;C\left(\int_{\Gamma_{h_+}(A,B)}|v_-|^2ds+\int_{\Gamma_{h_-}(A,B)}|v_+|^2ds\right. \\
	  &\qquad\left.+\hat R_-(A)+\hat R_-(B)+\hat R_+(A)+\hat R_+(B)\right), 
	\end{align}
	where 
	\ben
	  \hat R_\pm(P):=\int_{l_\pm(P)}|\pa_1v_\pm\pa_2\ov v_\pm|+|\ov v_\pm\pa_1v_\pm|ds,~P=A,B. 
	\enn
	Also we have $|\hat R_\pm(P)|$ are uniformly bounded for $P\in(-\infty,+\infty)$. It then follows that $-\varphi=v_-|_\Gamma\in L^2(\Gamma)$ and thus 
	\be\label{4.26}
	  \varphi(x)\rightarrow0~~\text{as}~|x_1|\rightarrow+\infty. 
	\en
	From (\ref{4.24}) and (\ref{4.26}), by the same arguments in the proof of \cite[Lemma 4.6]{DTS03}, we can prove that $v_\pm(x)\rightarrow0$ as $|x|\rightarrow\infty$ uniformly in $\ov{D^\pm_*}$, which implies 
	\ben
	  \lim\limits_{P\rightarrow\pm\infty}|\wid R_\pm(P)|=0,~~j=1,2. 
	\enn
	In view of (\ref{4.23}), we get $v_-=0$ on $\Gamma_{h_+}$ and $v_+=0$ on $\Gamma_{h_-}$. Moreover, by the boundary condition on $\Gamma_{h_\pm}$ in (\ref{4.22}), we obtain that $\pa_2v_-=0$ on $\Gamma_{h_+}$ and $\pa_2v_+=0$ on $\Gamma_{h_-}$, which indicates by Holmgren's uniqueness theorem that $v_\pm=0$ in $D^\pm_*$. Hence, $u_-=0$ in $D^+_{h_+}$ and $u_+=0$ in $D^-_{h_-}$. Finally, by the jump relations (\ref{4.19}) and (\ref{4.20}), it follows that $\varphi=\psi=0$ on $\Gamma$, which ends the proof. 
\end{proof}

With Theorems \ref{thm3.3}, \ref{thm4.6} and the results in \cite[Section 4.3]{DTS03}, we now state our main theorem for the TSP. 
\begin{theorem}\label{thm4.7}
	Suppose the scatterer $(\Gamma,k_-,k_+,\mu)$ satisfies $k_\pm,\mu>0$, $(\mu-1)(k_+^2-k_-^2\mu)\geq0$ and $k_+^2\neq k_-^2\mu$. Given $g_1\in BC^{1,\alpha}(\Gamma)$ and $g_2\in BC^{0,\alpha}(\Gamma)$, for $h_-,h_+\in\R$ with $h_-<f^-<f^+<h_+$, the TSP has exactly one solution $(u_+,u_-)$ in the form of
	\begin{align*}
	u_+(x)&=(\mathcal{D}_{k_+,h_-}^+\varphi)(x)+(\mathcal{S}_{k_+,h_-}^+\psi)(x),~~x\in D^+, \\ 
	u_-(x)&=\mu^{-1}[(\mathcal{D}_{k_-,h_+}^-\varphi)(x)+(\mathcal{S}_{k_-,h_+}^-\psi)(x)],~~x\in D^-, 
	\end{align*}
	such that $u_\pm\in C^2(D^\pm)\cap C^1(\ov D^\pm)\cap BC^1(\ov D^\pm_{h_\pm})$. Here, $\chi:=(\varphi,\psi)^T\in BC^{1,\alpha}(\Gamma)\times BC^{0,\alpha}(\Gamma)$ is the unique solution to the boundary integral equation 
	\be\label{4.14}
	\mathcal{M}\chi=G 
	\en
	with 
	\ben
	\mathcal{M}=\left(
	\begin{array}{cc}
		K_{k_+,h_-}^+-\mu^{-1}K_{k_-,h_+}^-+(1+\mu)(2\mu)^{-1}I & S_{k_+,h_-}^+-\mu^{-1}S_{k_-,h_+}^-  \\
		T_{k_-,h_+}^--T_{k_+,h_-}^+ & K_{k_-,h_+}^{\prime-}-K_{k_+,h_-}^{\prime+}+I \\
	\end{array}
	\right) 
	\enn
	and $G=(g_1,g_2)^T$, where the integral operator $\mathcal{M}$ is bijective (hence boundedly invertible) in $[BC(\Gamma)]^2$. Moreover, in $D^\pm_{h_\pm}$, $u_\pm$ depend continuously on $\|g_1\|_{\infty,\Gamma}$ and $\|g_2\|_{\infty,\Gamma}$, while $\na u_\pm$ depend continuously on $\|g_1\|_{1,\alpha,\Gamma}$ and $\|g_2\|_{0,\alpha,\Gamma}$. 
\end{theorem}
\begin{remark}\label{remark4.8}
	From {\rm \cite{TSK03}} we see that the unique solvability of the integral equation {\rm (\ref{4.14})} in $L^p$ $(1\leq p\leq\infty)$ spaces can be obtained from the unique solvability in the space of bounded and continuous functions. In other words, the integral opeator $\mathcal{M}$ is also bijective and boundedly invertible in $[L^p(\Gamma)]^2$. 
\end{remark}

\section{Singularity of solution}\label{sec5}
\setcounter{equation}{0}
In this section, we aim to derive the singularity of the solutions to the TSP corresponding to the incident point source or hypersingular point source in the $H^1$ sense, which is of great significance in the uniqueness proof of the inverse problem and also interesting on their own rights. Here we give a novel and simple perspective to investigate on the singularity of the solutions with singular incidence, which can be applied to many other cases. 

\begin{lemma}\label{lem5.1}
	For fixed $y\in\R^2$ and $k,k_1\geq0$, $\Phi_{k_1}(\cdot,y)-\Phi_{k}(\cdot,y)\in W^{2,p}(B_1(y))$ with $1\leq p<\infty$. 
\end{lemma}
\begin{proof}
	Since $\Phi_{k_1}-\Phi_{k}=\Phi_{k_1}-\Phi_0+\Phi_0-\Phi_{k}$, it suffices to show that $\Phi_{k}(\cdot,y)-\Phi_{0}(\cdot,y)\in W^{2,p}(B_1(y))$ for $k>0$. First we remind two recurrence formulas for the Hankel functions, i.e., 
	\ben
	  H^{(1)\prime}_0(z)=-H^{(1)}_1(z)~~~{\rm and}~~~H^{(1)}_1(z)+zH^{(1)\prime}_1(z)=zH^{(1)}_0(z). 
	\enn
	Then by direct calculation we see that
	\ben
	  \pa_i(\Phi_{k}(x,y)-\Phi_{0}(x,y))=\frac{x_i-y_i}{|x-y|}\left(\frac{1}{2\pi}\frac{1}{|x-y|}-\frac{i}{4}kH_1^{(1)}(k|x-y|)\right), 
	\enn
	and 
	\ben
	  &&\quad\pa_{ij}(\Phi_{k}(x,y)-\Phi_{0}(x,y)) \\ [1mm]
	  &&=\frac{\delta_{ij}|x-y|^2-2(x_i-y_i)(x_j-y_j)}{|x-y|^3}\left(\frac{1}{2\pi}\frac{1}{|x-y|}-\frac{i}{4}kH_1^{(1)}(k|x-y|)\right) \\
	  &&\quad-\frac{i}{4}\frac{(x_i-y_i)(x_j-y_j)}{|x-y|^2}k^2H_0^{(1)}(k|x-y|), 
	\enn
	where $i,j=1,2$ and $\delta_{ij}$ is the Kronecker symbol. Further, we recall that as $z\rightarrow0^+$(see details in \cite{NL72})
	\ben
	  \left|\frac{i}{4}H_0^{(1)}(z)-\frac{1}{2\pi}\ln\frac{1}{z}\right|\leq C,~~{\rm and }~~\left|\frac{1}{2\pi z}-\frac{i}{4}kH_1^{(1)}(kz)\right|\leq Cz|\ln(z)|, 
	\enn
	which implies that for $x\in B_1(y)$ and $i,j=1,2$, 
	\ben
	  &&|\Phi_{k}(x,y)-\Phi_{0}(x,y)|\leq C,~~~|\pa_i(\Phi_{k}(x,y)-\Phi_{0}(x,y))|\leq C, \\
	  &&|\pa_{ij}(\Phi_{k}(x,y)-\Phi_{0}(x,y))|\leq C|\ln|x-y||. 
	\enn
	Thus, the assertion follows. 
\end{proof}
\begin{theorem}\label{thm3.1}
	Suppose the scatterer $(\Gamma,k_-,k_+,\mu)$ is under the assumptions in Theorem {\rm \ref{thm4.7}}. Suppose further $\Gamma\in C^{2,\lambda}$ $(or~f\in C^{2,\lambda}(\R))$ with $\lambda\in (0,1]$.  For $x_0\in\Gamma$ and $\delta>0$ small enough, define $x_j:=x_0-(\delta/j)n(x_0)\in D^+$, $j\in\N$. Then the unique solutions $(u^{(j)}_+,u^{(j)}_-)$ to the TSP with the incident point source $u^i_j(x)=\Phi_{k_+}(x,x_j),j\in\N$, satisfy that 
	\ben
	\left\|u^{(j)}_--\frac{2}{\mu+1}\Phi_0(x,x_j)\right\|_{H^{1}(D_0)}\leq C, 
	\enn
	where the constant $C>0$ is independent of $j\in\N$, and $D_0$ is a $C^{2,\lambda}$ bounded domain such that $B_{2\delta}(x_0)\cap D^-\subset D_0\subset D^-$.  
\end{theorem}
\begin{proof}
By Lemma \ref{lem2.4} and Remark \ref{remark2.5}, we see that $\Phi_{k_-}(x,x_j)$, $j\in\N$, satisfy the DPRC for $k_-$ in $D^-$, which implies that $(u^{(j)}_+,u^{(j)}_--\mu^{-1}\Phi_{k_-}(x,x_j))$ is the unique solution to the TSP with the boundary data 
\ben
g_{1,j}=\mu^{-1}\Phi_{k_-}(x,x_j)-\Phi_{k_+}(x,x_j), \\
g_{2,j}=\frac{\pa\Phi_{k_-}(x,x_j)}{\pa n(x)}-\frac{\pa\Phi_{k_+}(x,x_j)}{\pa n(x)}.
\enn
By the regularity derived in Lemma \ref{lem5.1}, the Sobolev embedding and the decay property of the Hankel functions at infinity, we have $\|g_{2,j}\|_{0,\alpha,\Gamma}\leq C$ uniformly for $j\in\N$. Choose $R>0$ large enough such that $\ov D_0\subset B_R(x_0)$. Then there exists a cut off function $\chi\in C_0^\infty(\R^2)$, $0\leq\chi\leq1$ and 
\begin{align*}
	\chi(x)=\left\{
	\begin{array}{ll}
		0, & {\rm in}~\R^2\setminus\ov B_{R+1}(x_0), \\[2mm]
		1, & {\rm in}~B_R(x_0). 
	\end{array}
	\right.
\end{align*}
We can also find a $C^{2,\alpha}$ bounded domain $\Om$ satisfying that $B_{R+1}(x_0)\cap D^-\subset\Om\subset D^-$ and $\Gamma\cap\pa\Om=\Gamma\cap B_{R+1}(x_0)$. Define $(w_{1,j},w_{2,j})$ as the solutions to the TSP with $\tilde g_{1,j}=(\mu^{-1}-1)\chi(x)\Phi_{k_+}(x,x_j)$ and $\tilde g_{2,j}=0$. Clearly, $\|g_{1,j}-\tilde g_{1,j}\|_{1,\alpha,\Gamma}\leq C$ uniformly for $j\in\N$, which yields that $\|u^{(j)}_--\mu^{-1}\Phi_{k_-}(x,x_j)-w_{2,j}\|_{H^1(D_0)}\leq C$ by Theorem \ref{thm4.7}. 

Further, we know that $w_{2,j}$ has the form 
$$w_{2,j}(x)=\mu^{-1}[(\mathcal{D}_{k_-,h}^-\varphi_j)(x)+(\mathcal{S}_{k_-,h}^-\psi_j)(x)],~~x\in D^-, $$
for some $h\in\R$ with $-h<f^-<f^+<h$, where $\Phi_j=(\varphi_j,\psi_j)^T$ satisfies $\mathcal{M}\Phi_j=G_j$ with $G_j=(\tilde g_{1,j},\tilde g_{2,j})^T$. Direct calculation shows that $\|\tilde g_{1,j}\|_{L^p(\Gamma)}\leq C$ for $1\leq p<\infty$, which indicates by remark \ref{remark4.8} that $\varphi_j,\psi_j$ are uniformly bounded in $L^p(\Gamma)$ with $1\leq p<\infty$ for $j\in\N$. Now define $\widetilde w_{2,j}(x)=\mu^{-1}[\mathcal{D}_{k_-,h}^-(\chi\varphi_j)(x)+\mathcal{S}_{k_-,h}^-(\chi\psi_j)](x)$, $x\in D^-$. For a moment, we extend $\varphi_j,\psi_j$ to be zero in $\R^2\setminus\Gamma$ and still denote by $\varphi_j,\psi_j$. Then to be specific, 
\begin{align*}
\widetilde w_{2,j}(x)\;&=\frac{1}{\mu}\left(\int_{\pa\Om}\frac{\pa}{\pa n(y)}G_{k_-}^-(x,y;h)\chi(y)\varphi_j(y)ds(y)\right.\\
&\qquad\quad\left.+\int_{\pa\Om}G_{k_-}^-(x,y;h)\chi(y)\psi_j(y)ds(y)\right), ~~x\in D^-. 
\end{align*}
We can also obtain that 
\begin{align*}
w_{2,j}(x)-\widetilde w_{2,j}(x)\;&=\frac{1}{\mu}\left(\int_{\Gamma\setminus B_R(x_0)}G_{k_-}^-(x,y;h)(1-\chi(y))\varphi_j(y)ds(y)\right.\\
&\qquad\quad\left.+\int_{\Gamma\setminus B_R(x_0)}\frac{\pa}{\pa n(y)}G_{k_-}^-(x,y;h)(1-\chi(y))\psi_j(y)ds(y)\right)
\end{align*}
for $x\in D^-$. Thus, by (\ref{2.8}) and the positive distance between $\Gamma\setminus B_R(x_0)$ and $D_0$, we deduce that for $x\in\ov D_0$ 
\begin{align*}
|w_{2,j}(x)-\widetilde w_{2,j}(x)|\;&\leq C\int_{\Gamma\setminus B_R(x_0)}(1+|x_1-y_1|)^{-3/2}(|\varphi_j(y)|+|\psi_j(y)|)ds(y) \\
&\leq C\int_{\Gamma\setminus B_R(x_0)}(|\varphi_j(y)|+|\psi_j(y)|)ds(y)\\
&\leq C(\|\varphi_j\|_{L^1(\Gamma)}+\|\psi_j\|_{L^1(\Gamma)})\leq C
\end{align*}
with the constant $C>0$ independent of $x$ and $j\in\N$. Using (\ref{2.8}) again and arguing analogously as above, we can also obtain that $|\na( w_{2,j}-\widetilde w_{2,j})(x)|\leq C$ for $x\in\ov D_0$, which implies $\|w_{2,j}-\widetilde w_{2,j}\|_{H^1(D_0)}\leq C$ uniformly for $j\in\N$. 

Now we investigate $\widetilde w_{2,j}$. By the bounded embedding from $L^2(\pa\Om)$ into $H^{-1/2}(\pa\Om)$ and \cite[Corollary 3.8]{DR13}, we derive that 
$$\left\|\int_{\pa\Om}\Phi_{k_-}(x,y)\chi(y)\psi_j(y)ds(y)\right\|_{H^1(D_0)}\leq C. $$
Since $|x-y_h'|\geq\varepsilon>0$ for all $x\in D_0$ and $y\in \Gamma\cap B_{R+1}(x_0)$, we have 
\ben
&&\left\|\int_{\Gamma\cap B_{R+1}(x_0)}\frac{\pa}{\pa {n(y)}}\left(\Phi_{k_-}(x,y_h')+P_{k_-}^-(x-y_h')\right)\chi(y)\varphi_j(y)ds(y)\right\|_{H^1(D_0)}\\
&&\;+\left\|\int_{\Gamma\cap B_{R+1}(x_0)}\left(\Phi_{k_-}(x,y_h')+P_{k_-}^-(x-y_h')\right)\chi(y)\psi_j(y)ds(y)\right\|_{H^1(D_0)}\leq C  
\enn
from the regularity of $\Phi_{k_-}(x,y_h')$ and $P_{k_-}^-(x-y_h')$, which further implies that 
$$\left\|\widetilde w_{2,j}(x)-\frac{1}{\mu}\int_{\pa\Om}\frac{\pa\Phi_{k_-}(x,y)}{\pa n(y)}\chi(y)\varphi_j(y)ds(y)\right\|_{H^1(D_0)}\leq C$$
uniformly for $j\in\N$. Due to the positive distance between $\pa\Om\setminus B_R(x_0)$ and $D_0$, there exists a $\varepsilon_0>0$ small enough such that $\pa\Om\setminus B_{R-\varepsilon_0}(x_0)$ also has a positive distance from $D_0$. In other words, $D_0\subset\subset B_{R-\varepsilon_0}(x_0)$. Then by the first equation in $\mathcal{M}\Phi_j=G_j$, we see that 
\begin{align*}
\varphi_j(x)-\frac{2(1-\mu)}{1+\mu}\chi(x)\Phi_{k_+}(x,x_j)&=\frac{2\mu}{1+\mu}\left(\mu^{-1}K_{k_-,h}^--K_{k_+,-h}^+\right)(\varphi_j)(x)\\
&\quad+\frac{2\mu}{1+\mu}\left(\mu^{-1}S_{k_-,h}^--S_{k_+,-h}^+\right)(\psi_j)(x),~x\in\Gamma. 
\end{align*}
Analyzing similarly as $w_{2,j}-\widetilde w_{2,j}$, it follows from the positive distance between $\Gamma\setminus B_R(x_0)$ and $\Gamma\cap B_{R-\varepsilon_0}(x_0)$ that 
\ben
&&\left\|\left(\mu^{-1}K_{k_-,h}^--K_{k_+,-h}^+\right)[(1-\chi)\varphi_j]\right\|_{H^{1/2}(\Gamma\cap B_{R-\varepsilon_0}(x_0))}\\
&&\;+\left\|\left(\mu^{-1}S_{k_-,h}^--S_{k_+,-h}^+\right)[(1-\chi)\psi_j]\right\|_{H^{1/2}(\Gamma\cap B_{R-\varepsilon_0}(x_0))}\leq C. 
\enn
By the bounded embedding from $L^2(\pa\Om)$ into $H^{-1/2}(\pa\Om)$ and \cite[Corollary 3.7]{DR13}, we obtain that 
\ben
&&\left\|\int_{\pa\Om}\left(\frac{1}{\mu}\frac{\pa\Phi_{k_-}(x,y)}{\pa n(y)}-\frac{\pa\Phi_{k_+}(x,y)}{\pa n(y)}\right)\chi(y)\varphi_j(y)ds(y)\right\|_{H^{1/2}(\pa\Om)}\\
&&\;+\left\|\int_{\pa\Om}\left(\mu^{-1}\Phi_{k_-}(x,y)-\Phi_{k_+}(x,y)\right)\chi(y)\psi_j(y)ds(y)\right\|_{H^{1/2}(\pa\Om)}\leq C. 
\enn
Moreover, because of $|x-y_{\pm h}'|\geq\varepsilon>0$ for all $x\in\Gamma\cap B_{R-\varepsilon_0}(x_0)$ and $y\in\Gamma\cap B_{R+1}(x_0)$, it is deduced that 
\ben
&&\left\|\int_{\Gamma\cap B_{R+1}(x_0)}\frac{\pa}{\pa {n(y)}}\left(\Phi_{k_\pm}(x,y_{\mp h}')+P_{k_\pm}^\pm(x-y_{\mp h}')\right)\chi(y)\varphi_j(y)ds(y)\right\|_{H^{1/2}(\Gamma\cap B_{R-\varepsilon_0}(x_0))}\\
&&\;+\left\|\int_{\Gamma\cap B_{R+1}(x_0)}\left(\Phi_{k_\pm}(x,y_{\mp h}')+P_{k_\pm}^\pm(x-y_{\mp h}')\right)\chi(y)\psi_j(y)ds(y)\right\|_{H^{1/2}(\Gamma\cap B_{R-\varepsilon_0}(x_0))}\leq C, 
\enn
which indicates that 
$$\left\|\varphi_j(x)-\frac{2(1-\mu)}{1+\mu}\Phi_{k_+}(x,x_j)\right\|_{H^{1/2}(\Gamma\cap B_{R-\varepsilon_0}(x_0))}\leq C$$
uniformly for $j\in\N$. Then from \cite[Theorem 6.13]{WM00}, we have 
\ben
&&\left\|\int_{\pa\Om}\frac{\pa\Phi_{k_-}(x,y)}{\pa n(y)}\chi(y)\left(\varphi_j(y)-\frac{2(1-\mu)}{1+\mu}\Phi_{k_+}(y,x_j)\right)ds(y)\right\|_{H^1(D_0)}\\
&&\leq C\left(\left\|\varphi_j-\frac{2(1-\mu)}{1+\mu}\Phi_{k_+}(\cdot,x_j)\right\|_{H^{1/2}(\Gamma\cap B_{R-\varepsilon_0}(x_0))}\right.\\
&&\qquad\quad+\left.\left\|\chi\left(\varphi_j-\frac{2(1-\mu)}{1+\mu}\Phi_{k_+}(\cdot,x_j)\right)\right\|_{H^{-2}(\pa\Om)}\right)\leq C, 
\enn
and hence 
$$\left\|w_{2,j}(x)-\frac{2(1-\mu)}{\mu(1+\mu)}\int_{\pa\Om}\frac{\pa\Phi_{k_-}(x,y)}{\pa n(y)}\chi(y)\Phi_{k_+}(y,x_j)ds(y)\right\|_{H^1(D_0)}\leq C. $$
Since $\{x_j\}_{j\in\N}$ has a positive distance from $\pa\Om\setminus\ov D_0$ and $\pa D_0\setminus\Gamma$, we derive 
$$\left\|w_{2,j}(x)-\frac{2(1-\mu)}{\mu(1+\mu)}\int_{\pa D_0}\frac{\pa\Phi_0(x,y)}{\pa n(y)}\Phi_0(y,x_j)ds(y)\right\|_{H^1(D_0)}\leq C$$ 
by the regularity of $\Phi_{k_\pm}-\Phi_0$. 

Now define
\ben
\wid f_j(x)=2\int_{\pa D_0}\frac{\pa\Phi_0(x,y)}{\pa n(y)}\Phi_0(y,x_j)ds(y),~~~x\in D_0.
\enn
Set $h_j(x)=\Phi_0(x,x_j)$. Then we see that
\ben
\left\{
\begin{array}{ll}
	\Delta (\wid f_j+h_j)=0~~~&{\rm in}~D_0,\\[3mm]
	\displaystyle\wid f_j(x)+h_j(x)=2\int_{\pa D_0}\frac{\pa\Phi_0(x,y)}{\pa n(y)}\Phi_0(y,x_j)ds(y)~~~&{\rm on}~\pa D_0.
\end{array}
\right.
\enn
Again by \cite[Corollary 3.7]{DR13} we obtain that the trace of $\wid f_j+h_j$ on $\pa D_0$ is uniformly bounded in $H^{1/2}(\pa D_0)$. Thus the classic elliptic regularity gives $\|\wid f_j+h_j\|_{H^1(D_0)}\leq C$ uniformly for $j\in\N$. It immediately follows that 
$$\left\|w_{2,j}(x)+\frac{1-\mu}{\mu(1+\mu)}\Phi_0(x,x_j)\right\|_{H^1(D_0)}\leq C.$$
Therefore, finally we deduce that 
$$\left\|u^{(j)}_--\left(\frac{1}{\mu}-\frac{1-\mu}{\mu(1+\mu)}\right)\Phi_0(x,x_j)\right\|_{H^1(D_0)}\leq C$$
uniformly for $j\in\N$, which is the desired estimate. 
\end{proof}
\begin{theorem}\label{thm3.2}
	Suppose the scatterer $(\Gamma,k_-,k_+,\mu)$ is under the assumptions in Theorem {\rm \ref{thm4.7}}. Suppose further $\Gamma\in C^{2,\lambda}$ $(or~f\in C^{2,\lambda}(\R))$ with $\lambda\in (0,1]$. Follow the notations $x_0,x_j$ and $D_0$ in Theorem {\rm \ref{thm3.1}}. Define $y_j:=x_0+(\delta/j)n(x_0)\in D^-$, $j\in\N$. Then the unique solutions $(\wid u^{(j)}_+,\wid u^{(j)}_-)$ to the TSP corresponding to the incident hypersingular point source $\wid u^i_j=\na_x\Phi_{k_+}(x,x_j)\cdot n(x_0)$ satisfy the estimate 
	$$\left\|\wid u^{(j)}_--\frac{2}{\mu+1}\na_x\Phi_0(x,x_j)\cdot n(x_0)-v_j\right\|_{H^1(D_0)}\leq C$$
	uniformly for $j\in\N$, where 
	$$v_j(x)=\frac{2\mu(1-\mu)}{(\mu+1)^2}\int_{\pa D_0}\frac{\pa\Phi_0(x,y)}{\pa n(y)}(\na_x\Phi_0(y,x_j)+\na_x\Phi_0(y,y_j))\cdot n(x_0)ds(y),~x\in D_0, $$
	with $\|v_j\|_{L^2(D_0)}\leq C$ uniformly in $j\in\N$. 
\end{theorem}
\begin{proof}
	We know from Lemma \ref{lem2.4} and Remark \ref{remark2.5} that $\na_x\Phi_{k_+}(x,y_j)\cdot n(x_0)$ satisfy the UPRC for $k_+$ in $D^+$ and $\na_x\Phi_{k_-}(x,x_j)\cdot n(x_0)$ satisfy the DPRC for $k_-$ in $D^-$. Then it can be verified that $(\wid u^{(j)}_++(1-\mu)(\mu+1)^{-1}\na_x\Phi_{k_+}(x,y_j)\cdot n(x_0),\wid u^{(j)}_--2(\mu+1)^{-1}\na_x\Phi_{k_-}(x,x_j)\cdot n(x_0))$ solves the TSP with the boundary data $g_{m,j}=\wid g_{m,j}+\hat g_{m,j}$, $m=1,2$, where 
	\ben
	&&\wid g_{1,j}=\frac{2}{\mu+1}\na_x\left(\Phi_{k_-}(x,x_j)-\Phi_{k_+}(x,x_j)\right)\cdot n(x_0), \\
	&&\wid g_{2,j}=\frac{2\mu}{\mu+1}\frac{\pa}{\pa n(x)}\na_x\left(\Phi_{k_-}(x,x_j)-\Phi_{k_+}(x,x_j)\right)\cdot n(x_0), 
	\enn
	and 
	\ben
	&&\hat g_{1,j}=\frac{1-\mu}{\mu+1}\na_x\left(\Phi_{k_+}(x,x_j)+\Phi_{k_+}(x,y_j)\right)\cdot n(x_0), \\
	&&\hat g_{2,j}=\frac{\mu-1}{\mu+1}\frac{\pa}{\pa n(x)}\na_x\left(\Phi_{k_+}(x,x_j)-\Phi_{k_+}(x,y_j)\right)\cdot n(x_0).
	\enn
	Direct calculation shows that $\|\hat g_{1,j}\|_{\infty,\Gamma}+\|\hat g_{2,j}\|_{\infty,\Gamma}\leq C$ uniformly for $j\in\N$. From the proof of Lemma \ref{lem5.1}, we see that $\|\wid g_{1,j}\|_{0,\alpha,\Gamma}+\|\wid g_{1,j}\|_{H^{1/2}(\pa D_0\cap\Gamma)}+\|\wid g_{2,j}\|_{L^p(\Gamma)}\leq C$ uniformly for $j\in\N$ with $1\leq p<\infty$. Therefore, arguing analogously as in the proof of Theorem \ref{thm3.1}, we would obtain the desired estimate. The detailed proof is thus omitted. 
	
	To show the estimate for $v_j$, we first set $\phi_j(x)=(\na_x\Phi_0(x,x_j)+\na_x\Phi_0(x,y_j))\cdot n(x_0)$, $x\in\pa D_0$. It then can be derived that 
	\ben
	&&\left\|\int_{\pa D_0}\frac{\pa\Phi_0(x,y)}{\pa n(y)}\phi_j(y)ds(y)\right\|_{L^2(D_0)}\\
	&&=\sup\limits_{g\in L^2,\|g\|_{L^2(D_0)}=1}\left|\int_{D_0}\left(\int_{\pa D_0}\frac{\pa\Phi_0(x,y)}{\pa n(y)}\phi_j(y)ds(y)\right)g(x)dx\right|\\
	&&=\sup\limits_{g\in L^2,\|g\|_{L^2(D_0)}=1}\left|\int_{\pa D_0}\left(\frac{\pa}{\pa n(y)}\int_{D_0}\Phi_0(x,y)g(x)dx\right)\phi_j(y)ds(y)\right|\\
	&&\leq\sup\limits_{g\in L^2,\|g\|_{L^2(D_0)}=1}\left\|\frac{\pa}{\pa n(y)}\int_{D_0}\Phi_0(x,y)g(x)dx\right\|_{L^1(\pa D_0)}\|\phi_j\|_{L^\infty(\pa D_0)}\\
	&&\leq C\sup\limits_{g\in L^2,\|g\|_{L^2(D_0)}=1}\|g\|_{L^2(D_0)}\|\phi_j\|_{L^\infty(\pa D_0)} \\
	&&\leq C 
	\enn
	uniformly for $j\in\N$, since the volume potential operator is bounded from $L^2(D_0)$ into $H^2(D_0)$ and the boundary trace operator is bounded from $H^2(D_0)$ into $L^1(\pa D_0)$. 
\end{proof}
\begin{remark}\label{remark5.5}
	It should be noticed that in Theorems {\rm \ref{thm3.1}} and {\rm \ref{thm3.2}}, the assumptions in Theorem {\rm \ref{thm4.7}} are  used to only ensure the existence of the singular solutions and have nothing to do with the proof of the singularity. 
\end{remark}

\section{The inverse problem}\label{sec6}
\setcounter{equation}{0}
This section is devoted to the uniqueness proof of the inverse problem. More precisely, we want to identify the scatterer $(\Gamma,k_-,k_+,\mu)$ given $k_+$ and the measurements of the scattered field with the incident point source. In particular, the measurements are taken only on a line segment
$$\Gamma_{c,d}:=\{x=(x_1,x_2)\in\Gamma_c||x_1|\leq d\}$$
with the measurement width $d>0$ and measurement height $c>f^+$. With the previous results on well-posedness and singularity, our proof would be a standard application of the novel technique first proposed in \cite{JBH18}, which is mainly based on constructing a well-posed interior transmission problem in a sufficiently small domain and then using the singularity of the solutions to derive a contradiction. 

Let $D^+_m, D^-_m$ be the upper and lower half planes, respectively, separated by the interface $\Gamma_{f_m}$, $m=1,2$. Denote by $(u_{m,+}(\cdot,y),u_{m,-}(\cdot,y))$ the scattered wave and transmitted wave to the TSP associated with the scatterer $(\Gamma_{f_m},k_{m,-},k_+,\mu_m)$ and the incident point source $u^i(\cdot,y)=\Phi_{k_+}(\cdot,y)$ for $y\in D^+_m$, $m=1,2$. Now we present our main theorem for the inverse problem in this paper. 
\begin{theorem}\label{thm4.1}
	Suppose the scatterers $(\Gamma_{f_m},k_{m,-},k_+,\mu_m)$, $m=1,2$ are under the assumptions in Theorem {\rm \ref{thm3.1}}. If $u_{1,+}(x,y)=u_{2,+}(x,y)$ for all $x,y\in\Gamma_{c,d}$, then $(\Gamma_{f_1},k_{1,-},\mu_1)=(\Gamma_{f_2},k_{2,-},\mu_2)$. 
\end{theorem}
\begin{proof}
	Following the same procedure in the proof of \cite[Theorem 4.1]{JJB22}, from $u_{1,+}(x,y)=u_{2,+}(x,y)$ for all $x,y\in\Gamma_{c,d}$ we can get that $u_{1,+}(x,y)=u_{2,+}(x,y)$ for all $x,y\in D^+_1\cap D^+_2$. In the following proof, we show $k_{1,-}=k_{2,-}$ in the last part. 
	
	(i)We first consider the case that $\mu_1\neq1$ and $\mu_2\neq1$. We want to prove $\Gamma_{f_1}=\Gamma_{f_2}$ and $\mu_1=\mu_2$ in this case. Suppose $\Gamma_{f_1}\neq\Gamma_{f_2}$, we can find a point $x_0\in\Gamma_{f_1}\setminus\Gamma_{f_2}$. There is no loss of generality if we further let $f_1(x_0)>f_2(x_0)$. Then there exists $\delta>0$ small enough such that $B_{2\delta}(x_0)\cap\Gamma_{f_2}=\emptyset$. Define $$x_j:=x_0-\frac{\delta}{j}n(x_0)\in B_{2\delta}(x_0),~j\in\N. $$
	Since $f_1\in C^{2,\alpha}(\R)$, there exists a small $C^{2,\lambda}$ bounded domain $D_0$ such that $B_{2\delta}(x_0)\cap D_1^-\subset D_0\subset D_1^-\setminus\ov D_2^-$. 
	
	For simplicity, we set $(u^{(j)}_{m,+},u^{(j)}_{m,-})=(u_{m,+}(\cdot,x_j),u_{m,-}(\cdot,x_j))$ and $u_{m,j}=\Phi_{k_+}(\cdot,x_j)+u^{(j)}_{m,+}$ in $D^+_m$ for $m=1,2$ and $j\in\N$. It is seen that $(u^{(j)}_{1,-},u_{2,j})$ satisfy the following interior transmission problem 
	\be\label{4.1}
	\left\{
	\begin{array}{ll}
		\Delta V_j-V_j=h_{1,j}~~~&{\rm in}~D_0,\\[2mm]
		\Delta W_j-W_j=h_{2,j}~~~&{\rm in}~D_0,\\[2mm]
		V_j-W_j=g_{1,j}~~~&{\rm on}~\pa D_0,\\[3mm]
		\displaystyle\rho\frac{\pa V_j}{\pa n}-\frac{\pa W_j}{\pa n}=g_{2,j}~~~&{\rm on}~\pa D_0,
	\end{array}
	\right.
	\en
	with $(V_j,W_j)=(u^{(j)}_{1,-},u_{2,j})$, $\rho=\mu_1$ and 
	\ben
	&&h_{1,j}=-(k_{1,-}^2+1)u^{(j)}_{1,-},\quad h_{2,j}=-(k_+^2+1)u_{2,j}, \\
	&&g_{1,j}=(u^{(j)}_{1,-}-u_{2,j})|_{\pa D_0},\quad g_{2,j}=\left.\left(\displaystyle\mu_1\frac{\pa u^{(j)}_{1,-}}{\pa n}-\frac{\pa u_{2,j}}{\pa n}\right)\right|_{\pa D_0}. 
	\enn
	Clearly, $g_{1,j}=g_{2,j}=0$ on $\widetilde\Gamma:=\pa D_0\cap\Gamma$ by the boundary conditions in the TSP. Since $\mu_1\neq1$, by \cite[Theorem 6.7]{FD06} we obtain that 
	\begin{align}\nonumber
	\|u^{(j)}_{1,-}\|_{H^1(D_0)}+\|u_{2,j}\|_{H^1(D_0)}\leq &\;C\left(\|u^{(j)}_{1,-}\|_{L^2(D_0)}+\|u_{2,j}\|_{L^2(D_0)}\right.\\ \label{4.2}
	&\qquad\left.+\|g_{1,j}\|_{H^{1/2}(\pa D_0)}+\|g_{2,j}\|_{H^{-1/2}(\pa D_0)}\right)
	\end{align}
	with $C>0$ independent of $j\in\N$. It follows from Theorem \ref{thm3.1} and the uniform boundedness of $\Phi_0(x,x_j)$ in $L^2(D_0)$ that $\|u^{(j)}_{1,-}\|_{L^2(D_0)}\leq C$. Due to the positive distance between $B_{2\delta}(x_0)$ and $\Gamma_{f_2}$, we deduce that $\|\Phi_{k_+}(x,x_j)\|_{1,\alpha,\Gamma_{f_2}}+\|\pa_n\Phi_{k_+}(x,x_j)\|_{0,\alpha,\Gamma_{f_2}}\leq C$ and thus $\|u^{(j)}_{2,+}\|_{H^1(D_0)}\leq C$ uniformly for $j\in\N$ by Theorem \ref{thm4.7}, which implies $\|u_{2,j}\|_{L^2(D_0)}\leq C$. Then we show that $g_{1,j}$ and $g_{2,j}$ are uniformly bounded in $H^{1/2}(\pa D_0)$ and $H^{-1/2}(\pa D_0)$ for $j\in\N$, respectively. In fact, by Theorem \ref{thm3.1}, it can be derived that 
	\begin{align*}
	\|g_{1,j}\|_{H^{1/2}(\pa D_0)}+\|g_{2,j}\|_{H^{-1/2}(\pa D_0)}&=\|g_{1,j}\|_{H^{1/2}(\pa D_0\setminus\widetilde{\Gamma})}+\|g_{2,j}\|_{H^{-1/2}(\pa D_0\setminus\widetilde{\Gamma})}\\
	&\leq C\left(\|u^{(j)}_{1,-}\|_{H^1(D_0\setminus\ov{B_{2\delta}(x_0)})}+\|u^{(j)}_{2,+}\|_{H^1(D_0)}\right.\\
	&\qquad\quad\left.+\|\Phi_0(x,x_j)\|_{H^1(D_0\setminus\ov{B_{2\delta}(x_0)})}\right)\\
	&\leq C. 
	\end{align*}
	with $C>0$ independent of $j\in\N$. 
	
	Therefore, by (\ref{4.2}) we have $\|u_{2,j}\|_{H^1(D_0)}\leq C$ uniformly for $j\in\N$, which is a contradiction since $\|u^{(j)}_{2,+}\|_{H^1(D_0)}\leq C$ and $\|\Phi_{k_+}(x,x_j)\|_{H^1(D_0)}\rightarrow\infty$ as $j\rightarrow\infty$. Thus $\Gamma_{f_1}=\Gamma_{f_2}=\Gamma$. 
	
	Next we show $\mu_1=\mu_2$. Suppose not, then $\mu_1/\mu_2\neq1$.  It can be verified that $(u^{(j)}_{1,-},u^{(j)}_{2,-})$ is the solution to the interior tranmission problem \eqref{4.1} with $(V_j,W_j)=(u^{(j)}_{1,-},u^{(j)}_{2,-})$, $\rho=\mu_1/\mu_2$ and 
	\ben
	&&h_{1,j}=-(k_{1,-}^2+1)u^{(j)}_{1,-},\quad h_{2,j}=-(k_{2,-}^2+1)u^{(j)}_{2,-}, \\
	&&g_{1,j}=(u^{(j)}_{1,-}-u^{(j)}_{2,-})|_{\pa D_0},\quad g_{2,j}=\left.\left(\displaystyle\frac{\mu_1}{\mu_2}\frac{\pa u^{(j)}_{1,-}}{\pa n}-\frac{\pa u^{(j)}_{2,-}}{\pa n}\right)\right|_{\pa D_0}. 
	\enn
	Also, $g_{1,j}=g_{2,j}=0$ on $\widetilde{\Gamma}$. Arguing analogously as above, again from \cite[Theorem 6.7]{FD06} and Theorem \ref{thm3.1}, it is obtained that $\|u^{(j)}_{1,-}\|_{H^1(D_0)}\leq C$ uniformly for $j\in\N$, which however contradicts with Theorem \ref{thm3.1}. Hence, $\mu_1=\mu_2$. 
	
	(ii)Now we consider the case $\mu_1=1$ or $\mu_2=1$. Without loss of generality, we let $\mu_1=1$. Suppose $\Gamma_{f_1}\neq\Gamma_{f_2}$. Then we can find a $x_0$ such that one of the following is satisfied: \\ (a)$x_0\in\Gamma_{f_1}\setminus\Gamma_{f_2}$ and  $f_1(x_0)>f_2(x_0)$; 
	(b)$x_0\in\Gamma_{f_2}\setminus\Gamma_{f_1}$ and  $f_2(x_0)>f_1(x_0)$. \\
	As in (i) we also have $x_0,\delta,x_j$ and $D_0$ with the same properties. Denote by $(\widetilde u^{(j)}_{m,+},\widetilde u^{(j)}_{m,-})$ the solutions to the TSP with respect to the incident hypersingular point source $\widetilde u^i_j=\na_x\Phi_{k_+}(x,x_j)\cdot n(x_0)$ for $m=1,2$ and $j\in\N$. Set $\wid u_{m,j}=\na_x\Phi_{k_+}(x,x_j)\cdot n(x_0)+\wid u^{(j)}_{m,+}$ in $D^+_m$.  Again from the proof of \cite[Theorem 4.1]{JJB22}, we can get that $\wid u^{(j)}_{1,+}(x)=\wid u^{(j)}_{2,+}(x)$ for all $x\in D^+_1\cap D^+_2$ and $j\in\N$. Note that this also holds when $\mu_1=\mu_2\neq1$. 
	
	In case (a), analog to (\ref{4.1}), $(\wid u^{(j)}_{1,-},\wid u_{2,j})$ solves the following another type interior transmission problem 
	\be\label{4.4}
	\left\{
	\begin{array}{ll}
		\Delta\wid u^{(j)}_{1,-}+k_{1,-}^2\wid u^{(j)}_{1,-}=0~~~&{\rm in}~D_0,\\[2mm]
		\Delta\wid  u_{2,j}+k_+^2\wid u_{2,j}=0~~~&{\rm in}~D_0,\\[2mm]
		\wid u^{(j)}_{1,-}-\wid u_{2,j}=\hat g_{1,j}~~~&{\rm on}~\pa D_0,\\[3mm]
		\displaystyle\frac{\pa\wid u^{(j)}_{1,-}}{\pa n}-\frac{\pa\wid u_{2,j}}{\pa n}=\hat g_{2,j}~~~&{\rm on}~\pa D_0,
	\end{array}
	\right.
	\en
	with $\hat g_{1,j}=u^{(j)}_{1,-}-u_{2,j}$ and $\hat g_{2,j}=\pa_n u^{(j)}_{1,-}-\pa_n u_{2,j}$ satisfying $\hat g_{1,j}=\hat g_{2,j}=0$ on $\wid\Gamma$. Define $u_j=(1-\chi_1)(\wid u^{(j)}_{1,-}-\wid u_{2,j})$, where $\chi_1$ is a cut off function such that $\chi_1\in C_0^\infty(\R^2)$, $0\leq\chi_1\leq1$ and 
	\begin{align*}
		\chi_1(x)=\left\{
		\begin{array}{ll}
			0, & {\rm in}~\R^2\setminus\ov{B_{2\delta}(x_0)}, \\[2mm]
			1, & {\rm in}~B_{3\delta/2}(x_0). 
		\end{array}
		\right.
	\end{align*}
	Then we see that $u_j\in H^1_\Delta(D_0)$ and $u_j=\hat g_{1,j}$, $\pa_n u_j=\hat g_{2,j}$ on $\pa D_0$. Since $k_+\neq k_{1,-}$ under the conditions in Theorem \ref{thm4.7}, by the classic results in \cite{FDH10}, we know that we can let $D_0$ be sufficiently small such that the problem (\ref{4.4}) is well-posed in $L^2(D_0)\times L^2(D_0)$ and 
	\ben
	\|\wid u^{(j)}_{1,-}\|_{L^2(D_0)}+\|\wid u_{2,j}\|_{L^2(D_0)}\leq C\|u_j\|_{H^1_\Delta(D_0)}. 
	\enn
	We claim now that $\|u_j\|_{H^1_\Delta(D_0)}\leq C$ uniformly for $j\in\N$. It suffices to show that $\|\wid u^{(j)}_{1,-}\|_{H^1(D_0\setminus\ov{ B_{3\delta/2}(x_0)})}+\|\wid u_{2,j}\|_{H^1(D_0\setminus\ov{ B_{3\delta/2}(x_0)})}\leq C$, which is a direct application of Theorem \ref{thm3.2} and the positive distance between $B_{2\delta}(x_0)$ and $\Gamma_{f_2}$. Hence, we get $\|\wid u^{(j)}_{1,-}\|_{L^2(D_0)}\leq C$ and thus $\|\na_x\Phi_0(x,x_j)\cdot n(x_0)\|_{L^2(D_0)}\leq C$ by Theorem \ref{thm3.2}, which is a contradiction. Therefore, $\Gamma_{f_1}=\Gamma_{f_2}=\Gamma$. Then following similarly as the arguments in (i), we can obtain that $\mu_2=\mu_1=1$. 
	
	In case (b), if $\mu_2\neq 1$, for $(u_{1,j},u_{2,-}^{(j)})$ we can construct interior transmission problem in the form of (\ref{4.1}) to derive a contradiction. Thus in this case we still have $\mu_2=\mu_1=1$. Then arguing analogously as in case (a) we can deduce that $\Gamma_{f_1}=\Gamma_{f_2}=\Gamma$. The detailed proof is omitted. 
	
	(iii)Finally we show $k_{1,-}=k_{2,-}$. Suppose $k_{1,-}\neq k_{2,-}$. Again we see that $(\wid u^{(j)}_{1,-},\wid u^{(j)}_{2,-})$ satisfy the interior transmission problem similar to \eqref{4.4} 
	where the boundary data $\hat g_{1,j}=\wid u^{(j)}_{1,-}-\wid u^{(j)}_{2,-}$ and $\hat g_{2,j}=\pa_n\wid u^{(j)}_{1,-}-\pa_n\wid u^{(j)}_{2,-}$ with $\hat g_{1,j}=\hat g_{2,j}=0$ on $\widetilde{\Gamma}$. Following the previous argument in (ii), we then still get the contradiction that $\|\na_x\Phi_0(x,x_j)\cdot n(x_0)\|_{L^2(D_0)}\leq C$ uniformly for $j\in\N$ by Theorem \ref{thm3.2}. 
	Thus, $k_{1,-}=k_{2,-}$. The proof is finally complete. 
\end{proof}

\section*{Acknowledgements}
This work was supported by the NNSF of China with grant 12122114.

\end{document}